\documentclass{article}    
\usepackage{amsmath, amssymb, amsthm,pdfsync}
\usepackage{verbatim}
\usepackage{graphicx}
\usepackage{subfig, caption}
\usepackage{cancel}
\usepackage{enumitem}
\usepackage{wrapfig}
\usepackage{mathtools}
\usepackage{color}
\usepackage{hyperref}
\usepackage{fullpage}
\usepackage{overpic, caption}
\usepackage{authblk}
\usepackage{cleveref}
\usepackage{dsfont}
\usepackage{bbm}

\usepackage[utf8]{inputenc} 

\setlist[itemize]{itemsep=-1mm}

\newtheorem{thm}{Theorem}
\newtheorem{lem}[thm]{Lemma}
\newtheorem{cor}[thm]{Corollary}

\newtheorem{rem}[thm]{Remark}

\newtheorem{conj}[thm]{Conjecture}

\let\phi\varphi

\newcommand{\vertiii}[1]{{\left\vert\kern-0.25ex\left\vert\kern-0.25ex\left\vert #1 
    \right\vert\kern-0.25ex\right\vert\kern-0.25ex\right\vert}}
\newcommand{\I}{\mathbbm{1}}
\newcommand{\E}{\mathbb{E}}

\newcommand{\R}{\mathbb{R}}
\newcommand{\Z}{\mathbb{Z}}

\DeclareMathOperator{\var}{var}
\DeclareMathOperator{\cov}{cov}
\DeclareMathOperator{\dist}{dist}

\DeclareMathOperator{\supp}{supp}

\numberwithin{equation}{section}
\numberwithin{thm}{section}

\begin{document}

\title{Equivalence of a mixing condition and the LSI in spin systems with infinite range interaction}

\author{Christopher Henderson\thanks{Unit\'e Math\'ematiques Pures et Appliqu\'ees, \'Ecole Normale Sup\'erieure de Lyon, christopher.henderson@ens-lyon.fr}  \mbox{ and} Georg Menz\thanks{Department of Mathematics, UCLA, gmenz@math.ucla.edu}}

%

\maketitle

\begin{abstract}
We investigate unbounded continuous spin-systems with infinite-range interactions.  We develop a new technique for deducing decay of correlations from a uniform Poincar\'e inequality based on a directional Poincar\'e inequality, which we derive through an averaging procedure. We show that this decay of correlations is equivalent to the Dobrushin-Shlosman mixing condition.  With this, we also state and provide a partial answer to a conjecture regarding the relationship between the relaxation rates of non-ferromagnetic and ferromagnetic systems. Finally, we show that for a symmetric, ferromagnetic system with zero boundary conditions, a weaker decay of correlations can be bootstrapped.
\end{abstract}


%
%

\section{Introduction}\label{s_introduction}

In this article we consider a system of continuous real-valued spins on a subset of the lattice~$\mathbb{Z}^d$. The system is described by its finite volume Gibbs measures~$\mu_{\Lambda}^{\omega}$ (for a precise definition see Section~\ref{s_setting} below). Here,~$\Lambda \subset \mathbb{Z}^d$ denotes the domain on which the Gibbs measure is active, and~$\omega$ denotes the boundary values. Recently, some progress has been made in transferring classical results for finite-range interaction to infinite-range interaction. For instance, classical results regarding the absence of a phase-transition for systems on a one-dimensional lattice when the interaction is finite-range have recently been shown to hold when the interactions decay algebraically of order $2+\alpha$ \cite{Zeg96, bblMenzNittka}.\\

In that spirit, we generalize in this article another classical result to infinite-range interaction.  In higher dimensions, the situation is murkier than the one described above; phase transitions can occur.  In other words, the behavior of the system may depend on the size of the lattice and the boundary conditions imposed.  In physics, the high temperature region of a system is characterized by a sufficient decay of correlations.  We show that the high temperature region, where phase transitions do not occur, can be characterized by a uniform spectral gap or a uniform log-Sobolev inequality.  The connection between mixing conditions on the Gibbs measure and the relaxation of the associated Glauber dynamics has a long history dating back to the 1970s~\cite{HS76,SZ92a}.  Stroock and Zegarlinski~\cite{SZ92a,SZ92b}, in the compact spin space setting, showed that when the interactions have finite range, the following three conditions are equivalent:
\begin{enumerate}
\item The finite volume Gibbs measures~$\mu_{\Lambda}^{\omega}$ satisfies the Dobrushin-Shlosman mixing condition.
\item The finite volume Gibbs measures~$\mu_{\Lambda}^{\omega}$ satisfy a logarithmic Sobolev inequality (LSI) uniform in the system size~$|\Lambda|$ and the boundary condition~$\omega$.
\item The finite volume Gibbs measures~$\mu_{\Lambda}^{\omega}$ satisfy a Poincar\'e inequality (PI) uniform in the system size~$|\Lambda|$ and the boundary condition~$\omega$.
\end{enumerate}
In the setting of a compact or discrete spin space, this result was generalized in 1995 by Laroche~\cite{La95} to infinite-range interaction. The case of an unbounded spin space is technically more challenging due to the loss of compactness. Therefore it took until 2001 when Yoshida~\cite{bblYoshida01} was able to obtain this type of result to unbounded spins with finite-range interaction. In this article we further generalize this statement to unbounded spins with infinite-range interaction.\\

We first describe what the conditions~(1)-(3) mean. 
We start by describing condition (1) regarding mixing conditions. There are a variety of different mixing conditions, see for example~\cite{DS1,DS87,DW90} or~\cite{Martinelli_notes}, and many of them are equivalent. Among the mixing conditions, the most famous is the Dobrushin-Shlosman mixing condition (see~\cite{DS1}). The Dobrushin-Shlosman mixing condition referred to above is known to be true for example when the interaction is weak enough. The common thread between all mixing conditions is that each condition quantifies the influence on the finite-volume Gibbs measures~$\mu_{\Lambda}^{\omega}$ of changing a spin value over large distances and therefore are describing a static property of~$\mu_{\Lambda}^{\omega}$. Mixing conditions are quite useful in the study of spin-systems. As we mentioned above, one can use them to describe the high-temperature regime of the system. In this article, we show that a uniform PI yields the following mixing condition: The covariance of functions of bounded support with respect to the finite-volume Gibbs measures $\mu_{\Lambda}^{\omega}$ decays algebraically of order~$d+\alpha$, $\alpha>0$, in the distance between the supports of each function. We show in Lemma~\ref{p_relation_mixing_conditions}, below, that this condition is equivalent to the Dobrushin-Shlosman mixing condition. For closing the loop of equivalences, we use the special case of spin-spin correlations decaying algebraically in the distance of the sites.\\

We now describe conditions (2) and (3).  The PI and the LSI are associated to properties of the Glauber dynamics of the system, the naturally associated diffusion process on the state space of the system.  Both yield exponential decay to equilibrium of this system, which is described by the Gibbs measure $\mu_{\Lambda}^\omega$ (for details see for example the introduction in~\cite{M_Dis}). The LSI was originally introduced by Gross~\cite{Gross} and yields the PI by linearization (see for example~\cite{bblLedoux}). The difference between the LSI and the PI is that the LSI yields exponential convergence with respect to relative entropies, whereas the PI only yields convergence with respect to variance. This difference becomes important when looking at high-dimensional situations or continuum limits i.e.~sending $|\Lambda| \to \infty$ (see for example~\cite{GORV,FM}). Here, the LSI is used to derive quantitative error bounds on the convergence, whereas the PI usually only yields qualitative convergence results. Deducing a uniform LSI is usually significantly harder than deducing a uniform PI. For example there are coupling methods available for the PI but not for the LSI. Those coupling methods have shown to be highly efficient (see e.g.~\cite{CW97, Chen98,Sam05,Ol09, LPW09, Eb13}).\\

The fact that the three conditions (1),~(2), and~(3) are equivalent is quite striking. On the one hand, this statement connects a static property of the equilibrium state of the system, namely the decay of correlations, to a dynamic property, namely, the relaxation to the equilibrium. On the other hand, it states that a uniform LSI is equivalent to a uniform PI.  In these situations it would be sufficient to deduce a uniform PI to obtain the much stronger uniform LSI, significantly simplifying the task.  \\

In our main result, namely Theorem~\ref{p_main_result}, we obtain the equivalence of the mixing condition, the LSI, and PI even in the case of algebraically decaying infinite-range interactions.  Another feature compared to \cite{bblYoshida01} is that we require fewer assumptions on the interaction terms and on the single-site potentials of the system.  Indeed, we allow non-ferromagnetic interactions and additionally the single-site potentials only have to be a bounded perturbation of a strictly convex potential; we do not require that the interaction is ferromagnetic and that the single-site potentials are strictly super-quadratic (cf.~Remark~\ref{r_cond_ssp_yoshida} from below). Another consequence of our main result is that an algebraic decay of spin-spin correlations yields an algebraic decay of correlations of general functions~$f$ and~$g$. \\

There is another important difference to existing results on discrete and compact spin spaces (cf.~\cite{SZ92a,SZ92b,La95}). To show that decay of correlations yields a uniform LSI, we cannot allow general multi-body interactions and have to restrict ourselves to a purely quadratic two-body interaction. The reason is that for this implication, we refer to the two-scale approach, which only considered those type of interactions (cf.~\cite{OR07,bblMenz13} and Remark~\ref{r_general_interaction_II} from below).  However, the argument that a uniform PI yields decay of correlations is quite robust and applies even in the case of general multi-body interactions (see Theorem~\ref{p_sg_dc}, Remark~\ref{r_general_interaction_I} below).\\

We restrict ourselves to the spin space~$\mathbb{R}$ for technical reasons. In principle, we do not expect any significant obstacles for generalizing the main result to more complicated spin spaces like~$\mathbb{R}^k$. However, one would have to generalize the two-scale approach to those spaces. A more challenging question is if our main result --even for finite range interaction-- could be generalized to interacting birth-and-death Markov chains on~$\mathbb{N}$ or~$\mathbb{Z}$. For those systems one would have to work with a modified version of the LSI (called MLSI) because the LSI does not hold (see for example~\cite{CDPP09} and references therein). The main technical challenge would be to establish a two-scale criterion for the MLSI in the sense of~\cite{OR07} or~\cite{JSTV04}, which is an open problem.\\

Below, we formulate a conjecture saying that the LSI constant for a non-ferromagnetic system is bounded from below by the LSI constant of the associated ferromagnetic system (see Conjecture \ref{conjecture} from below).  As an application of our main result, we deduce a partial answer to this conjecture (see~Corollary~\ref{c_conjecture}).\\

The second result, namely Theorem~\ref{p_bootstrap}, considers the situation of zero boundary values, ferromagnetic interaction, and symmetric single-site potentials. We show that in this situation any initial decay of correlations may be bootstrapped to the same order decay as the interaction terms. In other words, $\cov(x_i,x_j)$ and $M_{ij}$ decay in $|i-j|$ of the same order.   
This is possible due to a bootstrapping principle which, while well-known for discrete spins (cf.~for example~\cite{FS82}), is very recent in the continuous spin-value setting. It was first carried out for one-dimensional lattice systems in~\cite{bblMenzNittka}. Compared to the discrete spin-value setting there are complications due to the fact that the Lebowitz inequality yields extra terms that must be controlled. Those terms are controlled by a careful iteration procedure. In this article, we refine the techniques of~\cite{bblMenzNittka} in order to generalize this technique to lattices of arbitrary dimension. 
In addition, our proof yields improved decay of correlations compared to the results of~\cite{bblMenzNittka}.

%
%

\section{Main Results}\label{s_results}

\subsection{Setting}\label{s_setting}

Fixing a finite lattice $\Lambda \subset\Z^d$, we define the formal Hamiltonian for this system as
\begin{align}
  H (x) \stackrel{\text{def}}{=} \sum_{i} \psi_i(x_i) + s_i x_i +  \sum_{i,j} M_{ij} x_i x_j, \label{d_Hamiltonian}
\end{align}
where the spin-spin interactions~$M$ are diagonally dominant, i.e.~there is a constant~$\delta>0$ such that
\begin{align}\label{e_diag_dominant}
 \delta
  + \sum_{j\neq i} |M_{ij}| \leq M_{ii},
\end{align}
holds for all $i\in\Lambda$.  We also suppose that the single site potentials $\psi_i$ are bounded perturbations of convex potentials, i.e that we may decompose $\psi_i$ as
\[
	\psi_i = \psi_{i,b} + \psi_{i,c}
\]
such that there is a constant $C_{\ref{SSPot}}$ such that for all $r\in \R$ we have
\begin{equation}\label{SSPot}
\psi_{i,c}''(r) \geq 0, ~~~~ |\psi_{i,b}(r)| + |\psi_{i,b}'(r)| \leq C_{\ref{SSPot}}.
\end{equation}
We point out that, in some sense, the single sites are strictly convex due to the $M_{ii}$ term and~\eqref{e_diag_dominant}.  If not stated otherwise, we assume that the interaction decays algebraically, e.g.
\begin{equation}\label{e_alg_decay}
	|M_{ij}| \leq \frac{C_{\ref{e_alg_decay}}}{[1 + |i-j|]^{2d+\alpha_{\ref{e_alg_decay}}}},
\end{equation}
for some constants $C_{\ref{e_alg_decay}}, \alpha_{\ref{e_alg_decay}}>0$.
Finally, we call the system ferromagnetic if
\begin{equation}\label{e_ferromagnetic}
M_{ij} \leq 0, ~~~~ \text{ for all } i\neq j.
\end{equation}

\begin{rem}\label{r_cond_ssp_yoshida}
Previous results showing the equivalence of the mixing condition, the LSI, and the PI, assumed a stronger condition on the single-site potential~$\psi$.  Namely, in these results it was assumed that for any constant~$m>0$ there is a decomposition $\psi_i = \psi_{i,b} + \psi_{i,c}$ such that that there is a constant $C$ such that for all $r\in \R$ we have
\begin{equation}\label{SSPot_Yos}
\psi_{i,c}''(r) \geq m, ~~~~ |\psi_{i,b}(r)| + |\psi_{i,b}'(r)| \leq C.
\end{equation}
In addition, it is usually assumed that the system is ferromagnetic.
See, for instance Theorem~2.1~in~\cite{bblYoshida01}.
\end{rem}

\begin{rem}
For lattice sites, $i\in \Lambda$, we use the $\ell^\infty$ norm,
\[ |i | = \max\{|i_1|, |i_2|, \dots, |i_d|\}, ~~\text{ where } i = (i_1,i_2,\dots, i_d),\]
to define distance.  We note, however, that the results stated below hold for all metrics which come from a norm on $\R^d$ as all norms on $\R^d$ are equivalent.  We abuse notation and use the symbols $|\cdot|$ to mean both the absolute value for scalar quantities and the $\ell^\infty$ norm on $\Z^d$.
\end{rem}

Given a finite lattice $\Lambda \subset \Z^d$, a set of values $\omega = (\omega_i)_{i\notin \Lambda}$, and a set of values $(s_i)_{i\in \Lambda}$, we denote the measure on $\Lambda$ given by this Hamiltonian 
as $\mu_\Lambda^\omega$.  Specifically
\begin{equation}\label{e_GibbsMeasure}
d\mu_\Lambda^\omega
	\stackrel{\text{def}}{=} \exp\left\{- \sum_{i\in\Lambda}( \psi_i(x_i) + s_i x_i) - \sum_{i,j\in \Lambda} M_{ij} x_i x_j - 2 \sum_{i \in \Lambda, j\notin \Lambda} M_{ij} x_i \omega_j\right\} \frac{dx}{Z_\Lambda^\omega},
\end{equation}
where $Z_\Lambda^\omega$ is a constant to make this a probability measure.  We call $\omega$ the boundary conditions and $s$ the external field.  In general, in the sequel, we drop the super-script notation and simply denote the measure as $\mu_\Lambda$.  In order to ensure that $\mu_\Lambda$ is well-defined, we restrict to ``well-tempered'' boundary conditions.  In other words, we only consider $\omega$ such that
\[
	\max_{i\in \Lambda}\sum_{j\notin \Lambda} |M_{ij}| |\omega_j| < \infty.
\]

We work in the following function spaces
\[\begin{split}
L^2(\mu_\Lambda) &\stackrel{\text{def}}{=} \big\{f: \Z^d \to \R ~|~ f(x) = f(y) \text{ if } x_i = y_i \text{ for all } i \in \Lambda,\\
	&~~~~~~~\text{ and } \int |f(x)|^2 d\mu_\Lambda(x) < \infty\big\},\\
H^1(\mu_\Lambda) &\stackrel{\text{def}}{=} \left\{f: \Z^d \to \R ~|~f, \nabla f \in L^2(\mu_\Lambda)\right\}
.\end{split}\]
Here, we have defined $\nabla f \stackrel{\text{def}}{=} (\nabla_i f)_{i\in\Lambda}$, where $\nabla_i$ is the derivative of $f$ at site $i$.  These spaces may be constructed in the usual way.  We often denote these simply as $L^2$ and $H^1$ when there is no chance of misinterpretation.  In practice, in this paper we always work with smooth functions whose derivatives are bounded, but we lose no generality in doing this since we may extend by density.  In any case, the typical example of a function that we care about, $f(x) = x_i$ for some site $i \in \Lambda$, falls into this class.  Lastly, we point out that the spaces $L^2(\mu_\Lambda)$ and $H^1(\mu_\Lambda)$ inherit a dependence on the boundary condition $\omega$ from $\mu_\Lambda$.

 \subsection{Main Results}\label{ss_main}

Our main goal in this paper is to generalize classical results about the equivalence of certain conditions to the setting of unbounded spin systems with algebraically decaying interactions (see in Section~\ref{s_introduction} and Theorem~\ref{p_main_result} from below). The first result states that a uniform PI yields mixing (cf.~Theorem~\ref{p_sg_dc} from below). For that purpose, let us first have a look at the different mixing conditions used in this article.\\

 We start with considering the first mixing condition, which is a statement about decay of correlations. More precisely, the mixing condition states that there are finite constants~$C_{\ref{e_stronger_mixing_condition}}>0$, $\alpha>0$, and~$\beta > 0$ such that for any finite subset~$\Lambda \subset \mathbb{Z}^d$, any boundary values $\omega$, and any functions~$f:  \Lambda  \to \mathbb{R}$ and~$g: \Lambda  \to \mathbb{R}$ with support $\supp f, \supp g \subset \Lambda$, it holds
    \begin{align}
      	|\cov_{\mu_\Lambda}(f, g)|
		& \leq C_{\ref{e_stronger_mixing_condition}} \frac{|\supp f|^\beta  \ |\supp g|^\beta}{[1+\dist(\supp f, \supp g)]^{d+ \alpha }} \   \ |\nabla f|_{L^{\infty}} \ |\nabla g|_{L^{\infty}}. \label{e_stronger_mixing_condition}
    \end{align}
Here,~$|\supp f|$ and~$|\supp g|$ denote the cardinality of the set~$\supp f$ and~$\supp g$ respectively. The condition~\eqref{e_stronger_mixing_condition} corresponds to the mixing condition (DS1) in~\cite{bblYoshida01}. \\

Now let us  introduce the second mixing condition, which is the analogue of the Dobrushin-Shlosman mixing condition in the continuous spin-value setting (see also Remark~2.2 in~\cite{bblYoshida01}). It states that there are finite constants~$C_{\ref{e_ds_1}} >0 $, $\alpha > 0$, and~$\beta > 0$ such that for any finite subset~$\Lambda \subset \mathbb{Z}^d$, and function~$f:  \supp f  \to \mathbb{R}$ with support $\supp f \subset \Lambda$, any site~$\mathbb{Z}^d \ni i \notin \Lambda$ and boundary data~$w, \tilde \omega  \in \mathbb{R}^{\mathbb{Z}^d \backslash \Lambda}$ with $\tilde \omega_k= \omega_k$ for all~$k \neq i$ it holds
    \begin{align}\label{e_ds_1}
      	\left| \int f \ d \mu_{\Lambda}^\omega   -  \int f \ d \mu_{\Lambda}^{\tilde \omega}  \right| \leq C_{\ref{e_ds_1}} \frac{ |\supp f|^\beta}{[1+\dist(\supp f, i)]^{d+ \alpha}} \ |\tilde \omega_k - \omega_k|   \  |\nabla f|_{L^{\infty}} .
 \end{align}

The following lemma shows that in our setting both mixing conditions~\eqref{e_stronger_mixing_condition} and~\eqref{e_ds_1} are  equivalent.
\begin{lem}\label{p_relation_mixing_conditions}
We assume that the single site potentials~$\psi_i$ are a bounded perturbation of a convex single-site potential in the sense of~\eqref{SSPot} and the interaction is diagonally dominant in the sense of \eqref{e_diag_dominant}. If, additionally, the interaction~$M_{ij}$ decays algebraically of order~$2d + \alpha_{\ref{e_alg_decay}}$ (cf~\eqref{e_alg_decay}), then the conditions~\eqref{e_stronger_mixing_condition} and~\eqref{e_ds_1} are equivalent.
\end{lem}
The statement that~\eqref{e_stronger_mixing_condition} yields~\eqref{e_ds_1} is well-known from the case of finite-range interaction (see, for example, the proof~$(DS1)\Rightarrow(DS2)$ in~\cite{bblYoshida01}). The statement that the Dobrushin-Shlosman mixing condition~\eqref{e_ds_1} yields the decay of correlations~\eqref{e_stronger_mixing_condition} in the case of unbounded interaction is new to our knowledge. The proof of Lemma~\ref{p_relation_mixing_conditions} is stated in Section~\ref{s_mixing_conditions_proof} below.\\


Let us now turn to the first main result of this article. It states that a uniform PI yields decay of correlations.

\begin{thm}\label{p_sg_dc}
Let the spin-spin interactions decay as in~\eqref{e_alg_decay} e.g.
\begin{equation*}
	|M_{ij}| \leq \frac{C_{\ref{e_alg_decay}}}{[1 + |i-j|]^{2d+\alpha_{\ref{e_alg_decay}}}},
\end{equation*}
for some constants $C_{\ref{e_alg_decay}}, \alpha_{\ref{e_alg_decay}}>0$.  Suppose that the finite volume Gibbs measure $\mu_\Lambda$ satisfies a uniform PI in the sense that there is a constant $\varrho_{\ref{sg}}$, independent of $\Lambda$ and the boundary condition $\omega$, such that for~$f\in H^1(\mu_\Lambda)$, we have
	\begin{equation}\label{sg}
	\var_{\mu_\Lambda}(f(x))
		\leq \frac{1}{\varrho_{\ref{sg}}} \E_{\mu_\Lambda}\left[ \sum_{x\in \Lambda} |\nabla_x f|^2\right]
	.\end{equation}
Then

\begin{align}
      	|\cov_{\mu_\Lambda}(f, g)|
		& \leq C_{\ref{e_stronger_mixing_condition_L_2}}
			\frac{ |\supp f|^{\frac{1}{2}} |\supp g|^{\frac{1}{2}} }{[1+\dist(\supp f, \supp g)]^{d+ \alpha_{\ref{e_alg_decay}}/2}}
			\left( \int |\nabla f|^2 \ d \mu_{\Lambda} \right)^{\frac{1}{2}} \left( \int |\nabla g|^2 \ d \mu_{\Lambda} \right)^{\frac{1}{2}}. \label{e_stronger_mixing_condition_L_2}
    \end{align}
In particular by Lemma \ref{p_relation_mixing_conditions}, the mixing condition~\eqref{e_stronger_mixing_condition} and~\eqref{e_ds_1} are satisfied.
\end{thm}

The proof of Theorem~\ref{p_sg_dc} is carried out in Section~\ref{s_sg_dc}. 
  \begin{rem}
    We want to point out that for the proof of Theorem~\ref{p_sg_dc}, we neither need the assumptions~\eqref{SSPot} on the single-site potentials~$\psi_i$ nor the assumption~\eqref{e_diag_dominant} on the diagonal dominance of the interaction~$M_{ij}$. We use the diagonal dominance simply to guarantee that $Z_\mu^\omega$ is finite in~\eqref{e_GibbsMeasure}.  The only assumptions we use in the course of the proof are the uniform PI in the sense of~\eqref{sg} and the decay~\eqref{e_alg_decay} of the interaction~$M_{ij}$.
  \end{rem}

\begin{rem}[General interactions I]\label{r_general_interaction_I}
The argument for Theorem~\ref{p_sg_dc} does not need that the interactions are purely quadratic two-body interactions and could easily be adapted to more general interactions. The only assumption on the interactions that is needed is that the mixed derivatives~$\nabla_i \nabla_j H$ of the Hamiltonian~$H$ satisfy the bound
\begin{align}\label{e_general_interaction_condition}
  	\sup_x |\nabla_i\nabla_j H(x)| \leq \frac{C_{\ref{e_alg_decay}}}{[1 + |i-j|]^{2d+\alpha_{\ref{e_alg_decay}}}}.
\end{align}
\end{rem}\bigskip

Now, we turn to the second step in our cycle of equivalences, namely that mixing yields a uniform LSI. For this we do not need that the mixing condition~\eqref{e_stronger_mixing_condition} is satisfied for all functions~$f$ and~$g$, but only for the functions~$f(x)=x_i$ and~$g(x)=x_j$ for all sites~$i, j \in \Lambda$. That decay of spin correlations implies a uniform LSI follows from the work in~\cite{bblMenz13}:


\begin{thm}[Theorem~1.7,~\cite{bblMenz13}]\label{p_otto_rez_rev}
Assume that the single-site potentials~$\psi_i$ are a bounded perturbation of a convex potential in the sense of~\eqref{SSPot} and that the interaction~$M_{ij}$ is diagonally dominant in the sense of~\eqref{e_diag_dominant}. Let the spin-spin interactions and the spin-spin correlations decay as
\begin{align}
	|M_{ij}| &\leq \frac{C_{\ref{e_menz}}}{[1 + |i-j|]^{d+\alpha_{\ref{e_menz}}}},\text{ and} \notag \\
	\cov_{\mu_\Lambda}(x_i,x_j) &\leq \frac{C_{\ref{e_menz}}}{[1 + |i-j|]^{d+\alpha_{\ref{e_menz}}}} \label{e_menz}
\end{align}
for some constants $C_{\ref{e_menz}}, \alpha_{\ref{e_menz}}>0$.  Then the finite volume Gibbs measures $\mu_\Lambda$ satisfy a uniform LSI in the sense that there is a constant $\rho_{\ref{e_lsi}}>0$, independent of $\Lambda$ and the boundary condition $\omega$, such that
\begin{equation}\label{e_lsi}
	\int f \log\left(\frac{f}{\int f d\mu_\Lambda}\right)d\mu_\Lambda \leq \frac{1}{2\rho_{\ref{e_lsi}}} \int \frac{|\nabla f|^2}{f} d\mu_\Lambda,
\end{equation}
for all smooth $f\geq 0$.
\end{thm}
\noindent We note that the inequality~\eqref{e_lsi} is referred to as the LSI.\\

\begin{rem}[General interactions II]\label{r_general_interaction_II}
The argument of~\cite{bblMenz13} to deduce Theorem~\ref{p_otto_rez_rev} uses at several steps that the interactions are of a purely quadratic two-body type. It would be very interesting to know how the argument of~\cite{bblMenz13} could be adapted to general interactions satisfying a similar bound as~\eqref{e_general_interaction_condition}.
\end{rem}

Finally, that the uniform LSI inequality implies the uniform PI is a straightforward linearization argument (see for example~\cite{bblLedoux}).  Hence, a combination of Theorem~\ref{p_sg_dc} and Lemma~\ref{p_relation_mixing_conditions} shows the equivalence of the three conditions discussed in Section~\ref{s_introduction}, which is summarized in the following statement.

\begin{thm}\label{p_main_result}
 Assume that the single-site potentials~$\psi_i$ are a bounded perturbation of a convex potential in the sense of~\eqref{SSPot}. Additionally, assume that that the interaction~$M_{ij}$ is diagonally dominant in the sense of~\eqref{e_diag_dominant} and the interactions~$M_{ij}$ decay algebraically of order~$2d+ \alpha_{\ref{e_alg_decay}}$ (cf.~\eqref{e_alg_decay}). \newline
Then the following are equivalent:
\begin{enumerate}
\item The finite volume Gibbs measures $\mu_\Lambda$ satisfy a uniform LSI  in the sense of~\eqref{e_lsi}.
\item The finite volume Gibbs measures $\mu_\Lambda$ satisfy a uniform PI  in the sense of~\eqref{sg}.
\item The decay of correlations condition~\eqref{e_stronger_mixing_condition} is satisfied i.e.~correlations decay algebraically of order~$d+ \alpha$ for some $\alpha>0$.
\item The Dobrushin-Shlosman mixing condition~\eqref{e_ds_1} is satisfied.
\item The spin-spin correlations decay algebraically of order~$d+ \alpha$ for some $\alpha>0$ (cf.~\eqref{e_menz}).
\end{enumerate}
\end{thm}

We want to note that our result holds up to an algebraic decay of the interaction~$M_{ij}$ of the order~$2d+\alpha_{\ref{e_alg_decay}}$. This is consistent with the result for compact spin spaces by Laroche~\cite{La95}, which needs the same order of decay. With high certainty, this is the best possible order of decay in one dimensional spin systems. For continuous spin systems one expects to have the same phase transition at a decay of the order~$2$ as in the Ising model (\cite{CaFeMePr,Dyson, FroeSpenc,Imbrie,Ruelle_1}). This phase transition is not yet established for continuous spins. However, R.~Nittka together with one of the authors deduced in~\cite{bblMenzNittka} the decay of correlation provided the interaction decays like~$2+ \alpha$, which together with the equivalence result of this paper directly yields the existence of a uniform LSI. This extends results of Chazottes~\cite{CCR08} on the validity of a uniform PI in one dimension.\\

An algebraic decay of the interaction~$M_{ij}$ of the order~$2d+\alpha$ for some $\alpha>0$ is not necessary for the validity of a uniform LSI. For example, as long as the interaction terms $M_{ij}$ are summable and diagonally dominant, in the sense of Definition~\ref{e_diag_dominant}, a simple application of~\cite[Theorem 1]{OR07} or of~\cite[Theorem~2]{Ma13} yields a uniform LSI for small enough interaction strength~$\max_i \sum_{j}|M_{ij}| \ll 1$.  We need the decay of the interaction~$M_{ij}$ of the order~$2d+\alpha$ to ensure the summability of certain terms (see also the comment after Lemma~\ref{p_site_estimate}).      \\

Section~\ref{s_sg_dc} is devoted to the proof of Theorem~\ref{p_sg_dc}.  The proof of Theorem~\ref{p_sg_dc} is not at all obvious; indeed, existing techniques all break down because of the relative strength of the interactions at long range due to the algebraic decay of the $M_{ij}$ terms. We develop a new approach as follows.  The first step is to use the Witten Laplacian (see~\cite{HelfferSjostrand, Sjostrand}) to recast the problem. Then in the second step, we deduce the decay of correlations via a directional Poincar\'e inequality (see Theorem~\ref{p_directional_poincare} and~\cite{bblMenz14}). Unfortunately, we cannot use the same technique to deduce the directional Poincar\'e inequality as in~\cite{bblMenz14} because it requires the off-diagonal interaction terms to be small compared to the constant in the Poincar\'e inequality, which we do not assume.  To overcome this obstacle, we employ a geometric averaging technique loosely inspired by the work of~\cite{bblBodineauHelffer}. The idea behind the averaging is to coarse grain the system such that we obtain a matrix inequality on blocks.  In order, for this to be successful, we have to ``blur'' the effect of the boundary between any two blocks by averaging over different choices of coarse graining (see Figure \ref{f_average}).\\

Let us now briefly remark on the applicability of the techniques of previous works to the proof of Theorem~\ref{p_sg_dc}. The argument of~\cite{La95} and~\cite{bblYoshida01} are based on the finite speed of propagation property of the Glauber dynamics of the system to obtain the analogue of Theorem~\ref{p_sg_dc} for finite range interactions. This argument obtains, from finite range interactions, an exponential decay of correlations.  In fact, this technique could be applied to interactions which decay algebraically, but would necessarily give algebraic decay of correlations of an uncontrollably small order. Since arguments which prove the log-Sobolev inequality from decay of correlations require the summability of the correlations, this is not enough to close the loop of equivalences in our setting.
There is another strategy, due to Bodineau and Helffer, that has been used to prove decay of correlations from the Poincar\'e inequality~\cite{bblBodineauHelffer}.  They use the Witten Laplacian and an averaging technique to deduce decay of correlations through a weighted $L^2$ estimate for the Green's function.  Their approach, applied to algebraically decaying interactions, fails for similar reasons as Yoshida's approach.  The order of decay of the correlations that their strategy obtains depends on balancing a number of parameters, and an application of this strategy, in our setting, does not necessarily yield decay of a high enough order to obtain summability of the correlations. Our approach is closer to that of Bodineau and Helffer but differs significantly on a technical level (see Section~\ref{ss_witten}).\\

We now formulate a conjecture, which is partially answered by Theorem~\ref{p_main_result}.
\begin{conj}\label{conjecture}
Given a Hamiltonian, $H$, with associated Gibbs measure $\mu_{\Lambda}$, satisfying the conditions in Section~\ref{s_setting}, consider a Hamiltonian, $H_{\text{fer}}$, with associated Gibbs measure $\mu_{\Lambda,\text{fer}}$, where the single-site potentials of $H$ and $H_{\text{fer}}$ are the same but where interaction terms of $H_{\text{fer}}$ are given by $-|M_{ij}|$ when $i\neq j$.  Suppose that $\mu_{\text{fer}}$ satisfies a LSI with constant $\varrho_{\text{fer}}$.  Then $\mu_{\Lambda}$ satisfies a LSI with constant $\varrho \geq \varrho_{\text{fer}}$.
\end{conj}
There are a number of facts indicating that this conjecture should be true.  Heuristically, the ferromagnetic situation is the worst case scenario for phase transitions.  For instance, the correlations of $\mu_{\Lambda}$ are always bounded from above by the correlations of $\mu_{\Lambda,\text{fer}}$.
\begin{lem}[Lemma~2.1, \cite{bblMenzNittka})]\label{p_cor_dom}
  Let $\mu_{\Lambda}$ and $\mu_{\Lambda,\text{fer}}$ be given as in Conjecture~\ref{conjecture}.  Then for any sites $i,j \in \Lambda$ we have
\[
	|\cov_{\mu_{\Lambda}}(x_i, x_j)| \leq \cov_{\mu_{\Lambda,\text{fer}}}(x_i,x_j).
\]
\end{lem}
In addition, it is easy to show that the conjecture is true in the simple case of a Gaussian Gibbs measure.  This fact follows easily from the fact that for Gaussian Gibbs measures, the positive definiteness of the Hamiltonian is equivalent to a LSI (see for example~\cite{OR07}).\\

Though we have no proof of this conjecture, even in the case of the spectral gap, Theorem~\ref{p_main_result} gives a weak version of this conjecture.  Namely, it yields the following corollary.
\begin{cor}\label{c_conjecture}
Assume that $\mu_{\Lambda,\text{fer}}$ is as in Conjecture~\ref{conjecture}, and assume that $\mu_{\Lambda,\text{fer}}$ satisfies a LSI with a constant uniform in the system size and the boundary condition.  Then $\mu_{\Lambda}$ also satisfies a LSI with a constant uniform in the system size and the boundary condition.
\end{cor}
\begin{proof}[Proof of Corollary~\ref{c_conjecture}]
If $\mu_{\Lambda,\text{fer}}$ satisfies a uniform LSI then, by Theorem~\ref{p_sg_dc} and Lemma~\ref{p_cor_dom}, we have that
\[
	|\cov_{\mu_{\Lambda}}(x_i,x_j)|
		\leq \cov_{\mu_{\Lambda,\text{fer}}}(x_i,x_j)
		\leq C (1 + |i - j|)^{-(d+\alpha_{\ref{e_alg_decay}}/2)}
.\]
Hence, by Theorem~\ref{p_otto_rez_rev}, it follows that $\mu_{\Lambda}$ satisfies a uniform LSI.
\end{proof}
\medskip

Our second main result addresses the case with zero boundary boundary values, ferromagnetic interaction, and symmetric single-site potentials. Our results allows to bootstrap initially weak decay in the correlations to the full decay of the spin-spin interactions~$M_{ij}$. For example in our argument above, the decay of the interaction terms $M_{ij}$ is an order $d$ larger than the decay of the spin-spin correlations. 
\begin{thm}\label{p_bootstrap}
Suppose that the Gibbs measure~$\mu_{\Lambda}$ has zero boundary values, ferromagnetic interaction and symmetric single-site potentials i.e. the Hamiltonian~$H$ given by~\eqref{d_Hamiltonian} satisfies~
\begin{align}
 \psi_i(z)=\psi_i(-z), \quad  s_i =0, \quad \mbox{and} \quad M_{ij} \leq 0 \qquad \mbox{for all sites $i \neq j$}. \label{e_Lebowitz_conditions}
\end{align}
Let the spin-spin interactions decay as
\begin{equation}\label{e_alg_decay_f}
	|M_{ij}| \leq \frac{C_{\ref{e_alg_decay_f}}}{(1 + |i-j|)^{d+\alpha_{\ref{e_alg_decay_f}}}},
\end{equation}
for some positive constants $C_{\ref{e_alg_decay_f}}$ and $\alpha_{\ref{e_alg_decay_f}}$.
Suppose that there are positive constants $\alpha_{\ref{FWDC}}$ and $C_{\ref{FWDC}}$ such that
\begin{equation}\label{FWDC}
	\cov_{\mu_\Lambda}(x_i, x_j)
		\leq \frac{C_{\ref{FWDC}}}{(1 + |i-j|)^{d + \alpha_{\ref{FWDC}}}}
,\end{equation}
Then there exists a constant $C_{\ref{p_bootstrap}}$ such that
\[
	\cov_{\mu_\Lambda}(x_i, x_j)
		\leq \frac{C_{\ref{p_bootstrap}}}{(1 + |i-j|)^{d + \alpha_{\ref{e_alg_decay_f}}}}
.\]
\end{thm}

Section~\ref{s_bootstrap}  is devoted to the proof of Theorem~\ref{p_bootstrap}.  The key ingredient here is the Lebowitz inequality along with an iterative procedure similar to the one used in~\cite{bblMenzNittka} on the one-dimensional lattice. Because our work here is more precise, Theorem~\ref{p_bootstrap} yields stronger decay of correlations than the original statement of~\cite{bblMenzNittka} for the one dimensional lattice.

\begin{rem}
 As in the work of Yoshida (cf.~Remark~2.4. in~\cite{bblYoshida01}), we stated our condition on the decay of correlations, the LSI, and the PI uniformly over all finite subsets~$\Lambda \subset \mathbb{Z}^d$. However, it is clear from our proofs that everything remains true if one considers only nice finite subsets~$\Lambda \in \mathcal{B}(n_0)$, where~$\mathcal{B}(n_0)$ denotes the set of generalized boxes with minimal side-length~$n_0$ (cf.~\cite{Yos99}). This rules out some pathological phenomena caused by sets~$\Lambda$ with irregular shapes (cf.~\cite{MO94} and~\cite{MO94b}).
\end{rem}

\paragraph{Outline of article:}

In Section~\ref{s_mixing_conditions_proof}, we state the proof of Lemma~\ref{p_relation_mixing_conditions}, showing the equivalence of the decay of correlations and the Dobrushin-Shlosman mixing condition. Section~\ref{s_sg_dc} is devoted to the proof of Theorem~\ref{p_sg_dc}, which allows us to show the equivalence of the PI, the LSI, and the decay of correlations. In Section~\ref{s_bootstrap} we give the proof of Theorem~\ref{p_bootstrap}.

%
%

\section{Proof of Lemma~\ref{p_relation_mixing_conditions}: Relationship between the different mixing conditions}\label{s_mixing_conditions_proof}

This section is devoted to the proof of Lemma~\ref{p_relation_mixing_conditions}.  To begin, we need uniform control on the variance of a single-site~$x_i$,~$i \in \Lambda$.
\begin{lem}\label{p_uniform moment estimate}
  Assume that the interaction is diagonally dominant in the sense of \eqref{e_diag_dominant} and the single site potentials~$\psi_i$ are bounded perturbations of convex single-site potentials in the sense of~\eqref{SSPot}.
Then there is a positive constant~$C_{\ref{e_unifrom_variance_estimate}}$ such that for any finite set~$\Lambda \subset \mathbb{Z}^d$, any boundary data~$\omega$, and any site~$i \in \Lambda$
  \begin{align}\label{e_unifrom_variance_estimate}
    \var_{\mu_{\Lambda}^\omega} (x_i) \leq C_{\ref{e_unifrom_variance_estimate}}.
  \end{align}
\end{lem}

\begin{proof}[Proof of Lemma~\ref{p_uniform moment estimate}.]
The proof consists of two steps. In the first step, we reduce the estimate of~$\var_{\mu_{\Lambda}^\omega} (x_i)$ to an estimate of the second moment of a symmetrized measure. This second moment is then estimated via a standard result.\\

We recall that the Hamiltonian~$H$ is given by
\begin{align*}
  H (x) =  \sum_{i\in\Lambda}( \psi_i(x_i) + s_i x_i) + \sum_{i,j\in \Lambda} M_{ij} x_i x_j + 2 \sum_{i \in \Lambda, j\notin \Lambda} M_{ij} x_i \omega_j.
\end{align*}
First we point out that
\[
	\var_{\mu_\Lambda^\omega}(x_i) 
	= \frac{1}{2} \int \int \left( x_i - y_i \right)^2 \mu_{\Lambda}^\omega (dx) \mu_{\Lambda}^\omega (dy).
\]
Then it holds by using~\eqref{e_GibbsMeasure} that
\begin{align*}
  \var_{\mu_{\Lambda}^\omega} (x_i) & = \frac{1}{2} \int \int \left( x_i - y_i \right)^2 \mu_{\Lambda}^\omega (dx) \mu_{\Lambda}^\omega (dy) \\
   & = \frac{1}{2} \int \int \left( x_i - y_i \right)^2 \frac{\exp\left\{- H(x) - H(y) \right\} }{\int \exp\left\{- H(x') - H(y')\right\} dx' dy'}   dx dy.
 \end{align*}
Using this and changing variables with $ x = p + q$ and $y = p - q$, we obtain
\begin{align*}
    &\var_{\mu^\omega_\Lambda}(x_i)
    	 = \frac{1}{2}  \int \int  p_i^2 \frac{\exp\left\{- H(p+q) - H(p-q) \right\}}{\int\int \exp\left\{- H(p'+q') - H(p'-q')\right\} dp' dq' }    dp dq  \\
    & ~ = \frac{1}{2}  \int \int  p_i^2 \frac{\exp\left\{- H(p+q) - H(p-q) \right\}}{\int \exp\left\{- H(p'+q) - H(p'-q)\right\} dp'} 
\left(
\frac{\int \exp\left\{- H(p'+q) - H(p'-q)\right\} dp'}{\int\int \exp\left\{- H(p'+q') - H(p'-q')\right\} dp' dq' }     \right)\ dpdq.  
\end{align*}
Notice that the term in parentheses integrates to 1 in the variable $q$.  Hence, given a uniform bound such as
\begin{align}\label{eq:nittka}
  \int  p_i^2 \frac{\exp\left\{- H(p+q) - H(p-q) \right\}}{\int \exp\left\{- H(p'+q) - H(p'-q)\right\} dp' }    dp \leq C,
\end{align}
we have our desired estimate.  However,~\eqref{eq:nittka} follows from Lemma~4.2 in~\cite{bblMenzNittka}, the proof of which depends only on the diagonal dominance of the interaction matrix and carries over as is in the $d>1$ case.  This finishes the proof.

\end{proof}

Equipped with the uniform estimate of Lemma~\ref{p_uniform moment estimate} on the variance of single-site spins, we can turn to the proof of Lemma~\ref{p_relation_mixing_conditions}.
\begin{proof}[Proof of Lemma~\ref{p_relation_mixing_conditions}]
  We start with showing that~\eqref{e_stronger_mixing_condition} yields~\eqref{e_ds_1}. For this, we adapt the short argument outlined in~\cite{bblYoshida01} to our setting. For this, let us consider two boundary values~$\omega$ and~$\tilde \omega$ that coincide except at one site~$k \notin \Lambda$. This means that for all~$  l \neq k$ it holds~$\omega_l = \tilde \omega_l$. Let us now consider for~$s \in [0,1]$ the interpolation
 \begin{align*}
\omega^{s} = s \tilde \omega + (1- s) \omega.   
 \end{align*}
Recalling that $\supp(f)\subset \Lambda$, we have that
 \begin{align}
  &\left| \int f \ d \mu_{\Lambda}^{\tilde \omega} -    \int f \ d \mu_{\Lambda}^{\omega} \right|
  	 =\left| \int_0^1 \frac{d}{ds} \int f \ d \mu_{\Lambda}^{\omega^s} ds \right| \notag
	 = \left| \int_0^1 \cov_{\mu_{\Lambda}^{\omega^s}} \left(f, \sum_{i \in \supp(f)} M_{ik} x_i \left(  \tilde \omega_k -  \omega_k  \right) \right) ds \right| \notag \\
	&~~ = \left| \tilde \omega_k -  \omega_k \right| \ \left| \int_0^1 \sum_{i \in \supp(f)} M_{ik} \cov_{\mu_{\Lambda}^{\omega^s}} (f, x_i   ) ds \right| \notag \\
	 &~~ \leq \left| \tilde \omega_k -  \omega_k \right| \ \left[ \left| \int_0^1 \sum_{i \in \supp f} M_{ik} \cov_{\mu_{\Lambda}^{\omega^s}} (f, x_i   ) ds \right|  + \left| \int_0^1 \sum_{i \in \Lambda, i \notin \supp f} M_{ik} \cov_{\mu_{\Lambda}^{\omega^s}} (f, x_i   ) ds \right| \right] . \label{e_deducing_ds1_step_1}
 \end{align}
 The second inequality is a straightforward computation using the form of $\mu_\Lambda^{\omega^s}$, and the third equality uses the linearity of the covariance.\\
 
Now, we estimate the covariance terms on the right hand side of the last identity with the help of the mixing condition~\eqref{e_stronger_mixing_condition}. More precisely, we get for the first sum that
\begin{align*}
  \sum_{i \in \Lambda, i \in \supp f} &  \left| M_{ik}  \cov_{\mu_{\Lambda}^{\omega^s}} (f,  x_i )  \right|  \\
&  =   \frac{1}{2} \sum_{i \in \Lambda, i \in \supp f} \ |M_{ik}|  \ \int \int \left(f(y) - f(x) \right) (y_i-x_i)  \mu_{\Lambda}^{\omega^s} (dx) \mu_{\Lambda}^{\omega^s} (dy)   \\
&  = \frac{1}{2}    \sum_{i \in \Lambda, i \in \supp f} \ |M_{ik}|  \ \int \int \sum_{j \in \supp f} \nabla_j f(\xi) \left(y_j - x_j \right) (y_i-x_i)  \mu_{\Lambda}^{\omega^s} (dx) \mu_{\Lambda}^{\omega^s} (dy),
\end{align*}
where $\xi$ is some point between $x$ and $y$.  Now introducing the $L^\infty$ norm of $f$ into the estimate yields
\begin{align}
  \sum_{i \in \Lambda, i \in \supp f} &  \left| M_{ik}  \cov_{\mu_{\Lambda}^{\omega^s}} (f,  x_i )  \right|  \\
 &  \leq   \frac{1}{2}  \sum_{i \in \Lambda, i \in \supp f}  |M_{ik}|  \| \nabla f \|_{L^{\infty}} \ \int \int \sum_{j \in \supp f}  \left|y_j - x_j \right| \ |y_i - x_i|  \mu_{\Lambda}^{\omega^s} (dx) \mu_{\Lambda}^{\omega^s} (dy)  \notag \\
 &  \leq   \frac{1}{4}  \sum_{i \in \Lambda, i \in \supp f}  |M_{ik}|  \| \nabla f \|_{L^{\infty}} \  \left( \sum_{j \in \supp f} \var_{\mu_{\Lambda}^{\omega^s}} \left(x_j \right)  + \var_{\mu_{\Lambda}^{\omega^s}} (x_i) \right) \notag \\
 &  \leq  \frac{C}{2}  \sum_{i \in \Lambda, i \in \supp f}  |M_{ik}|  \| \nabla f \|_{L^{\infty}} \  |\supp f| \  \notag \\
 &  \leq  \frac{C}{2}  \frac{1}{[1+\dist(\supp f, k)]^{d+ \alpha_{\ref{e_alg_decay}}}}   \| \nabla f \|_{L^{\infty}} \  |\supp f|^2 \label{e_deducing_ds1_case_1}
\end{align}
where $C$ is a universal constant given by the estimate~\eqref{e_unifrom_variance_estimate} of Lemma~\ref{p_uniform moment estimate} and where we used the decay of $M_{ik}$ given by the assumption~\eqref{e_alg_decay}.\newline
 
Let us now consider the second sum on the right hand side of~\eqref{e_deducing_ds1_step_1}.  Here, we apply our mixing condition~\eqref{e_stronger_mixing_condition} and we compute that
\begin{align}
 \sum_{i \in \Lambda, i \notin \supp f}  & \left |  \cov_{\mu_{\Lambda}^{\omega^s}} (f, M_{ik} x_i )  \right| \notag \\
& \leq  C \sum_{i \in \Lambda, i \notin \supp f}    |M_{ik}|   \frac{1}{[1+\dist(\supp f, i)]^{d+ \alpha_{\ref{e_alg_decay}}/2}} \ |\supp f|^{\beta} \|\nabla f\|_{L^{\infty}}  \notag \\
& \leq C \sum_{i \in \Lambda, i \notin \supp f} \frac{1}{[1+|i-k|]^{2d+ \alpha_{\ref{e_alg_decay}}}}  \frac{1}{[1+\dist(\supp f, i)]^{d+ \alpha_{\ref{e_alg_decay}}/2}} \ |\supp f|^{\beta} \|\nabla f\|_{L^{\infty}} \notag \\
& \leq  C  \ \frac{1}{[1+\dist(\supp f, k)  ]^{d+ \alpha_{\ref{e_alg_decay}}/2}}  \ |\supp f|^{\beta} \ |\nabla f|_{L^{\infty}} \sum_{i \in \Lambda, i \notin \supp f} \frac{1}{[1+|i-k|]^{d+ \frac{\alpha_{\ref{e_alg_decay}}}{2}}} \notag \\
& \leq  C \ \frac{1}{[1+\dist(\supp f, k)  ]^{d+ \alpha_{\ref{e_alg_decay}}/2}}  \ |\supp f|^{\beta} \ |\nabla f|_{L^{\infty}} \label{e_deducing_ds1_case_2} ,
\end{align}
where we used the assumption~\eqref{e_alg_decay} and the elementary estimate
\begin{align*}
  \frac{1}{[1+ | i - k |]^{d+ \alpha_{\ref{e_alg_decay}}/2}}\  \frac{1}{[1+\dist(\supp f, i)  ]^{d+ \alpha_{\ref{e_alg_decay}}/2}}  \leq \frac{1}{[1+\dist(\supp f, k)  ]^{d+ \alpha_{\ref{e_alg_decay}}/2}}.
\end{align*}
Here we point out that the last inequality in~\eqref{e_deducing_ds1_case_2} depends on the summability for $z\in\Z^d$ of a sum of the form $(1 + |z|)^{-(d+\alpha)}$ for $\alpha > 0$.  This is a key point that recurs in this article and is the reason for the hypothesis on the decay of the interaction terms $M_{ij}$~\eqref{e_alg_decay}.\\

Hence, a combination of the estimates~\eqref{e_deducing_ds1_step_1},~\eqref{e_deducing_ds1_case_1} and~\eqref{e_deducing_ds1_case_2} yields the desired estimate
 \begin{align*}
  &\left| \int f \ d \mu_{\Lambda}^{\tilde \omega} -    \int f \ d \mu_{\Lambda}^{\omega} \right| \leq  C \  \frac{1}{[1+\dist(\supp f, k)  ]^{d+ \alpha_{\ref{e_alg_decay}}/2}}  \ |\supp f|^{\beta} \ |\nabla f|_{L^{\infty}}.
 \end{align*}
 This concludes the proof that~\eqref{e_stronger_mixing_condition} yields~\eqref{e_ds_1}.\\

Let us now turn to the second implication, namely we will now show that~\eqref{e_ds_1} yields~\eqref{e_stronger_mixing_condition}. We start with rewriting the covariance~$\cov_{\mu_{\Lambda}}(f,g)$ as
\begin{align*}
  \cov_{\mu_{\Lambda}}(f,g) & = \frac{1}{2}\int (f(y) - f(x)) (g(y) - g(x)) \mu_{\Lambda}(dx) \mu_{\Lambda}(dy).
\end{align*}
Now, we condition on the spin values~$x_i$, $i \in \supp f$ in the support of f.  In other words, we write
\begin{align*}
  \frac{1}{2} & \int (f(y) - f(x)) (g(y) - g(x)) \mu_{\Lambda}(dx) \mu_{\Lambda}(dy) \\
& =  \frac{1}{2}  \int (f(y) - f(x)) \\
  & \qquad \times \left( \int g(y) \mu_{\Lambda} \left( d (y_j)_{\Lambda \backslash \supp f}|(y_i)_{ \supp f} \right)- \int g(x)\mu_{\Lambda} \left( d (x)_{\Lambda \backslash \supp f}|(x_i)_{\supp f} \right) \right) \mu_{\Lambda}(dx) \mu_{\Lambda}(dy).
\end{align*}
At this point we would like to apply the mixing condition~\eqref{e_ds_1} to the term 
\begin{align*}
  \left( \int g(y) \mu_{\Lambda} \left( d (y_j)_{\Lambda \backslash \supp f}|(y_i)_{ \supp f} \right)- \int g(x)\mu_{\Lambda} \left( d (x)_{\Lambda \backslash \supp f}|(x_i)_{\supp f} \right) \right).
\end{align*}
However, this cannot be done directly because the boundary data~$(y_i)_{ \supp f}$ and~$(x_i)_{\supp f}$ may differ at more than one site. This problem can be solved easily by interpolating between~$(y_i)_{ \supp f}$ and~$(x_i)_{\supp f}$. Without loss of generality we may index the set~$\supp f$ by the numbers~$1$ to~$M = |\supp f|$. This means that we can identify the set~$\supp f$ with
\begin{align*}
  \supp f = \left\{ 1 , \ldots, M  \right\}.
\end{align*}
Let~$z^k\in \mathbb{R}^{\supp f}$,~$k \in \left\{ 0 , \ldots, M  \right\} $, denote an interpolation of the spin values~$(x_i)_{i \in \supp f}$ and~$(y_i)_{i \in \supp f}$ i.e.
\begin{align*}
  z_l^k =
  \begin{cases}
    y_l , & \mbox{if } 1 \leq l \leq k \\
      x_l , & \mbox{if }  k < l \leq M.
  \end{cases}
\end{align*}
Then by using a telescope sum we can write
\begin{align*}
&  \int g(y) \mu_{\Lambda} \left( d (y_j)_{\Lambda \backslash \supp f}|(y_i)_{ \supp f} \right)- \int g(x)\mu_{\Lambda} \left( d (x)_{\Lambda \backslash \supp f}|(x_i)_{\supp f} \right) \\
  & \quad = \sum_{k=1}^M \int g(y) \mu_{\Lambda} \left( d (y_j)_{\Lambda \backslash \supp f}|z^k\right)- \int g(x)\mu_{\Lambda} \left( d (x)_{\Lambda \backslash \supp f}|z^{k-1} \right) .
\end{align*}
Now, we can apply the mixing condition~\eqref{e_ds_1} on each term in the sum on the right hand side of the last identity, which yields that
\begin{align*}
&  \sum_{k=1}^M \int g(y) \mu_{\Lambda} \left( d (y_j)_{\Lambda \backslash \supp f}|z^k\right)- \int g(x)\mu_{\Lambda} \left( d (x)_{\Lambda \backslash \supp f}|z^{k-1} \right) \\
& \qquad \leq  \sum_{k=1}^M C \frac{1}{[1+\dist(\supp g, k)]^{d+ \alpha_{\ref{e_alg_decay}}/2}} |y_k - x_k| \ |\supp g|^{\beta} \ |\nabla g\|_{L^{\infty}}  \\
& \qquad \leq C \frac{1}{[1+\dist(\supp f, \supp g)]^{d+ \alpha_{\ref{e_alg_decay}}/2}} \ |\supp g|^{\beta} \|\nabla g\|_{L^{\infty}} \ \sum_{k=1}^M  |y_k - x_k|  .
\end{align*}
Hence we obtain the inequality
\[
\cov_{\mu_{\Lambda}}(f,g)
	 \leq  \frac{1}{2} \ \frac{ |\supp g|^{\beta} }{[1+\dist(\supp f, i)]^{d+ \alpha_{\ref{e_alg_decay}}/2}} \|\nabla g\|_{L^{\infty}} \  \int |f(y) - f(x)| \sum_{k=1}^M |y_k - x_k|  \mu_{\Lambda}(dx) \mu_{\Lambda}(dy)
\]
Recall that, by the mean-value theorem,
\begin{align*}
  |f(y) - f(x)| & \leq \|\nabla f\|_{L^{\infty}} \ \left( \sum_{k=1}^M |y_k - x_k|^2 \right)^{\frac{1}{2}}.
\end{align*}
Hence, we obtain
\begin{align*}
  \cov_{\mu_{\Lambda}}(f,g) 
 & \leq  \frac{1}{2} \ \frac{ |\supp g|^{\beta} \|\nabla f\|_{L^{\infty}} \|\nabla g\|_{L^{\infty}}}{[1+\dist(\supp f, i)]^{d+ \alpha_{\ref{e_alg_decay}}/2}}  \int \left(  \sum_{k=1}^M |y_k - x_k|^2 \right)^{\frac{1}{2}}   \sum_{k=1}^M |y_k - x_k|  \mu_{\Lambda}(dx) \mu_{\Lambda}(dy) \\
 & \leq  \frac{1}{2} \ \frac{ |\supp g|^{\beta} \|\nabla f\|_{L^{\infty}} \|\nabla g\|_{L^{\infty}} }{[1+\dist(\supp f, i)]^{d+ \alpha_{\ref{e_alg_decay}}/2}} M^{\frac{1}{2}} \int \sum_{k=1}^M |y_k - x_k|^2    \mu_{\Lambda}(dx) \mu_{\Lambda}(dy)  \\
 & = \frac{1}{2} \ \frac{ |\supp g|^{\beta} \|\nabla f\|_{L^{\infty}} \|\nabla g\|_{L^{\infty}}}{[1+\dist(\supp f, i)]^{d+ \alpha_{\ref{e_alg_decay}}/2}} M^{\frac{1}{2}}  \sum_{k=1}^M \var_{\mu_{\Lambda}} ( x_k).
\end{align*}
Applying now the estimate~\eqref{e_unifrom_variance_estimate} of Lemma~\ref{p_uniform moment estimate} and considering that~$M=\supp f $ yields the desired inequality
\begin{align*}
  \cov_{\mu_{\Lambda}}(f,g) & \leq C \ \frac{1}{[1+\dist(\supp f, i)]^{d+ \alpha_{\ref{e_alg_decay}}/2}}  |\supp f|^{3/2} |\supp g|^\beta \|\nabla f\|_{L^{\infty}} \|\nabla g\|_{L^{\infty}}.
\end{align*}

\end{proof}

\section{Proof of Theorem~\ref{p_sg_dc}: Poincar\'e inequality implies decay of correlations}\label{s_sg_dc}

In order to obtain decay of correlations from the Poincar\'e inequality, we use ideas based on the work in~\cite{bblMenz14}.  The main result of that article is a covariance estimate that holds for weak interaction i.e.~if $C_{\ref{e_alg_decay}}$ is small. In the first step, we represent the covariance using the solution~$\varphi$ of Poisson's equation, see~\eqref{e_representation_covariance} from below. In the second step, we deduce the decay of correlations from a directional Poincar\'e inequality (cf.~Theorem~\ref{p_directional_poincare} from below). We explain those two steps in Section~\ref{ss_witten}, where we also discuss the similarities and differences of our approach and the approach of Bodineau and Helffer~\cite{bblBodineauHelffer}.\\

The main ingredient for the proof of Theorem~\ref{p_sg_dc} is the directional Poincar\'e inequality, Theorem~\ref{p_directional_poincare} from below. We state the proof of the directional Poincar\'e inequality in Section~\ref{ss_proof}. The argument consists of two main steps. In the first step, we derive from the uniform Poincar\'e inequality a weighted inequality for the gradient $\nabla \varphi$ (see Corollary~\ref{p_bl1} from below). After some elementary reordering and counting multiplicities, the directional Poincar\'e inequality is then deduced from Corollary~\ref{p_bl1} through some elementary linear algebra. The proof of the directional Poincar\'e inequality, i.e.~Theorem~\ref{p_directional_poincare}, requires two auxiliary lemmas, which are proved in Section~\ref{ss_lemmas}.\\

In order to simplify the computations, we use the following notation.

\begin{rem}
We define $a \lesssim b$ if there is a constant $C>0$ such that
\[ a\leq C b,\]
where $C$ depends only on the constants in \eqref{e_alg_decay}, \eqref{e_diag_dominant}, \eqref{SSPot}, \eqref{sg}, and the dimension.
\end{rem}

\subsection{ Proof of Theorem~\ref{p_sg_dc}}\label{ss_witten}

\subsubsection*{Representation of the covariance}

Our goal is to estimate $\cov_{\mu_\Lambda}(f, g)$ for two arbitrary functions~$f$ and~$g$ with finite support~$\supp f \subset \Lambda $ and~$\supp g \subset \Lambda$.  
The starting point of our argument is a representation of~$\cov_{\mu_\Lambda}(f, g)$ using the potential~$\varphi\in H^1(\mu_\Lambda)$, defined as the weak solution of
\begin{align}\label{e_d_varphi}
	-L\phi = f - \int fd\mu_\Lambda
\end{align}
where
\begin{align*}
  L \varphi = \Delta \varphi - \nabla H \cdot \nabla \varphi.
\end{align*}
%
%

In other words, we choose~$\varphi \in H^1 (\mu_\Lambda)$ such that if~$\xi \in H^1 (\mu_\Lambda)$, then
\begin{align} \label{e_weak_solution_elliptic}
  \int \nabla \xi \cdot \nabla \varphi \ d\mu_\Lambda = \int \xi \left(f - \int f \mu_{\Lambda} \right) \ d \ \mu_{\Lambda},
\end{align}
Notice that if $\phi$ and $\xi$ are smooth, then
\[
	- \int \xi L \varphi \mu_\Lambda  = 
  \int \nabla \xi \cdot \nabla \varphi \ d\mu_\Lambda.
\]

The problem above is easily seen to be well-posed as the functional corresponding to~\eqref{e_weak_solution_elliptic} is continuous and coercive due to the Poincar\'e inequality.  It follows from~\eqref{e_weak_solution_elliptic} that, almost everywhere,~\eqref{e_d_varphi} holds. 
Using the potential~$\varphi$ and the definition of $f$, we have the desired representation of~$\cov_{\mu_\Lambda} (f, g)$ as
\begin{equation}\begin{split}
  \label{e_representation_covariance}
  \cov_{\mu_\Lambda} (f, g) = \int \nabla \varphi  \ \cdot  \nabla g \ d \mu_\Lambda .
\end{split}\end{equation}

{ \subsubsection*{Comparison with work of Bodineau and Helffer} }

At this point, let us discuss the differences and similarities of our approach compared to the approach of Bodineau and Helffer~\cite{bblBodineauHelffer}.\\


Both approaches start with the same representation (cf.~\eqref{e_representation_covariance}) of the covariance i.e. (cf.~(3.12) in~\cite{bblBodineauHelffer}).  
In~\cite{bblBodineauHelffer}, the authors introduce a diagonal matrix~$M$ with certain weights and apply the Cauchy-Schwarz inequality to obtain:
\begin{align*}
  \cov_{\mu_{\Lambda}}(g,f) & = \int \nabla \varphi \cdot \nabla g \ d \mu_{\Lambda} \\
  & =  \int M\nabla \varphi \cdot M^{-1}\nabla g \ d \mu_{\Lambda} \\
  & =  \left( \int |M\nabla \varphi|^2 \ d\mu_{\Lambda} \right)^{\frac{1}{2}} \left( \int | M^{-1}\nabla g |^2 \ d \mu_{\Lambda} \right)^{\frac{1}{2}}.
\end{align*}
The main technical ingredient in~\cite{bblBodineauHelffer} is now a weighted~$L^2$ estimate (cf.~(3.13) in~\cite{bblBodineauHelffer}),
\begin{align}\label{e_tech_ingredient_BH}
  \int |M\nabla \varphi|^2 \ d\mu_{\Lambda} \leq C \int |M\nabla f|^2 \ d\mu_{\Lambda},
\end{align}
from which the desired decay of correlations follows by comparing the coefficients of the diagonal matrix~$M$ in the remaining terms.\\

In the present work, we proceed from the representation~\eqref{e_representation_covariance} of the covariance as follows. Because~$g(x)=g((x_i)_{\supp g})$ one gets
\begin{align*}
  \cov_{\mu_{\Lambda}}(g,f) & = \int \nabla \varphi \cdot \nabla g \ d \mu_{\Lambda} \\ 
  & = \int \sum_{i \in \supp g} \nabla_i \varphi  \nabla_i g \ d \mu_{\Lambda}.
\end{align*}
We do not introduce the diagonal matrix~$M$ and directly apply Cauchy-Schwarz. So, we get
\begin{align}\label{e_decay_correlations_cauchy_schwarz}
  \cov_{\mu_{\Lambda}}(g,f) & \leq \left(  \int \sum_{i \in \supp g} |\nabla_i \varphi|^2 d \mu_{\Lambda} \right)^{\frac{1}{2}} \left( \int \sum_{i \in \supp g} | \nabla_i g|^2 \ d \mu_{\Lambda} \right)^{\frac{1}{2}}.
\end{align}
\mbox{}\\

\subsubsection*{The directional Poincar\'e inequality and conclusion of the proof of Theorem~\ref{p_sg_dc}}
The main ingredient of the proof of Theorem~\ref{p_sg_dc} is a direct estimation of~$  \int |\nabla_i \varphi|^2 d \mu_{\Lambda}$ for~$i \in \supp g$. In Section~\ref{ss_proof} we will deduce the following statement:
\begin{thm}[Directional Poincar\'e inequality]\label{p_directional_poincare}
Under the the same assumptions as in Theorem~\ref{p_sg_dc}, let $\phi$ be given by the solution of~\eqref{e_d_varphi}. Then for any~$i \in \Lambda$,
  \begin{align}\label{e_tech_ingredient_HM}
    \int |\nabla_i \varphi|^2 d \mu_{\Lambda}  \lesssim \frac{|\supp f|}{\left( 1+ \dist (i,\supp f)\right)^{2d+\alpha_{\ref{e_alg_decay}}} } \int |\nabla f|^2 d \mu_{\Lambda}.
  \end{align}
\end{thm}

Using this estimate in~\eqref{e_decay_correlations_cauchy_schwarz} we get 
\begin{align*}
  \cov_{\mu_{\Lambda}}(g,f) & \lesssim   \frac{|\supp f|^{\frac{1}{2}} |\supp g|^{\frac{1}{2}}}{\left( 1+ \dist (\supp f , \supp g)\right)^{2d+\alpha_{\ref{e_alg_decay}}}}  \left(  \int  |\nabla f|^2 d \mu_{\Lambda} \right)^{\frac{1}{2}} \left( \int  | \nabla g|^2 \ d \mu_{\Lambda} \right)^{\frac{1}{2}} ,
\end{align*}
which is the desired estimate of Theorem~\ref{p_sg_dc}. Therefore one only has to show Theorem~\ref{p_directional_poincare} in order to proof Theorem~\ref{p_sg_dc}. \\

In the derivation of Lemma~\ref{p_directional_poincare}, we use a similar procedure as was used in the derivation of the estimate~\eqref{e_tech_ingredient_BH} in~\cite{bblBodineauHelffer}. A preliminary inequality is refined by an averaging procedure (see the proof of Theorem 7 in~\cite{bblBodineauHelffer} and Lemma~\ref{p_site_estimate} from below). However, on a technical level the derivation is quite different. Compared to~\cite{bblBodineauHelffer} we need more sophisticated estimates, but obtain better control of the decay of the correlations in relation to the decay of the interaction strength~$M_{ij}$. This allows us to consider the case of algebraically-decaying, infinite-range interaction.
\\

The proof of Theorem~\ref{p_directional_poincare} is given in Section~\ref{ss_proof}. Before proceeding to the proof, let us explain why we call the estimate~\eqref{e_tech_ingredient_HM} directional Poincar\'e inequality (see also~\cite[Theorem 2.7]{bblMenz14}). For this let us recall a well-known equivalent formulation of the Poincar\'e inequality that we will use in the sequel (see e.g.~\cite[Proposition 1.3]{bblLedoux} and~\cite[Section 7]{OttoVillani}).
\begin{lem}\label{p_dual_poincare}
A probability measure $\mu$ satisfies a Poincar\'e inequality with constant $\varrho$ if and only if for any $f$ and $\varphi$ solving \eqref{e_weak_solution_elliptic}, we have
\begin{align}\label{e_dual_PI}
	\int |\nabla \phi|^2 d\mu \leq \frac{1}{\varrho^2} \int |\nabla f|^2 d\mu.
\end{align}
\end{lem}
Comparing the estimates~\eqref{e_tech_ingredient_HM} and~\eqref{e_dual_PI}, one sees that~\eqref{e_tech_ingredient_HM} estimates each coordinate of the gradient of~$\phi$ separately. Summing up the estimate~\eqref{e_tech_ingredient_HM} over all coordinates yields the estimate~\eqref{e_dual_PI}. Therefore, the directional Poincar\'e inequality is a refinement of the equivalent formulation~\eqref{e_dual_PI} of the Poincar\'e inequality.

\subsection{Proof of Theorem~\ref{p_directional_poincare}: The directional Poincar\'e inequality by block averaging}\label{ss_proof}

This section is devoted to the proof of Theorem~\ref{p_directional_poincare}. As we have seen in Section~\ref{ss_witten}, this is the crucial ingredient for the proof of Theorem~\ref{p_sg_dc}. We deduce Theorem~\ref{p_directional_poincare} in two steps. In the first step, we deduce a weighted inequality for the gradient $\nabla \varphi$ (cf.~(2.14) in~\cite{bblMenz14} and~\eqref{e_lattice_ineq} from below). In the second step, the desired inequality is deduced through some elementary linear algebra (see \eqref{e_vector_ineq} ff.).\\


Because the interaction may not be small, i.e.~$C_{\ref{e_alg_decay}}$ maybe arbitrarily large, deriving the directional Poincar\'e inequality is very challenging. Hence, a naive derivation following the argument of~\cite{bblMenz14} fails (see discussion of Lemma~\ref{p_site_estimate} from below). Motivated by coarse-graining techniques (see for example~\cite{GORV}), it is a natural idea to circumvent this problem by deducing the desired weighted inequality on large blocks. However, these blocks still interact strongly with their neighbors at their boundaries. In order to deal with this, we average over all possible choices of blocks, which ``blurs'' the boundary between any two blocks, see Figure~\ref{f_average}. By this procedure one is able to compare the size of the bulk of the block, which scales like $R^d$, with the size of the boundary of the block, which scales like~$R^{d-1}$. Hence, one can control the interaction between nearest neighbor blocks choosing large enough blocks (see~\eqref{e_estimate_block_interaction} from below) and is able to deduce the desired weighted estimate~\eqref{e_lattice_ineq}. The remaining part of the proof, namely deducing the directional Poincar\'e inequality by some linear algebra, is then straightforward and unproblematic (see the argument after~\eqref{e_lattice_ineq}).



\subsubsection*{The idea of the proof}

In order to motivate what follows, we first show the proof of Theorem~\ref{p_directional_poincare} under the additional assumption
\begin{align} \label{e_ass_small_inter}
\max_{i}\sum_{j\neq i} |M_{ij}| \ll 1.
\end{align}
This means that the interaction is quite small between two different sites. In this case, the starting point of our proof is almost the same as in~\cite{bblMenz14} (cf.~the last equation before~\cite[(2.15)]{bblMenz14}). 

\begin{lem}\label{p_site_estimate}
Let $\phi$ be given by the solution of~\eqref{e_weak_solution_elliptic}. For any $l\in \Z^d$ with $l \notin \supp f $ it holds
\[\begin{split}
 &\left(\varrho_{\ref{sg}} -  \frac{1}{2}\sum_{j\neq l} |M_{lj}|\right)  \int |\nabla_l \varphi|^2 d\mu_\Lambda -  \sum_{j \neq l} \frac{|M_{lj}|}{2}  \int |\nabla_j \varphi|^2d\mu_\Lambda
		\leq 0.\end{split}\]
Additionally, it holds for any~$k \in \supp f$ that
\begin{align}\label{e_esti_nabla_k_0_sg}
  &\varrho_{\ref{sg}}^2  \int |\nabla_{k} \varphi|^2 d\mu_\Lambda  \leq 
  	\varrho_{\ref{sg}}^2  \int |\nabla \varphi|^2 d\mu_\Lambda  \leq \int | \nabla f|^2 d\mu_\Lambda.
\end{align}
\end{lem}
For the proof of Lemma~\ref{p_site_estimate} note that the first estimate of Lemma~\ref{p_site_estimate} is a special case of Corollary~\ref{p_bl0} from below and the second estimate~\eqref{e_esti_nabla_k_0_sg} is a direct consequence of the equivalent formulation of the Poincar\'e inequality, i.e.~Lemma~\ref{p_dual_poincare}. The difference between the proof of Lemma~\ref{p_site_estimate} and the proof of the key estimate in~\cite{bblMenz14} is that we apply Young's inequality in place of Cauchy-Schwarz (see~\eqref{e_averaged} f.~from below).  This seemingly minor difference appears to be unavoidable and crucial for our argument but implies that we need a decay of the interaction of the order~$2d+\alpha_{\ref{e_alg_decay}}$ and not just of the order $d+\alpha_{\ref{e_alg_decay}}$. \\

Under the additional assumption~\eqref{e_ass_small_inter} the number
\begin{align}\label{e_BL_diagonal_dominance}
	\left(\varrho_{\ref{sg}} - \max_{i}\sum_j|M_{ij}|  \right) >0
\end{align}
is strictly positive.  Define, momentarily,
\[
	\rho \stackrel{\rm def}{=} \left(\varrho_{\ref{sg}} - \frac{1}{2}\max_{i}\sum_j|M_{ij}|  \right)
\]
and notice that by \eqref{e_ass_small_inter}
\[
 \rho >0. 
\]
Using the matrix~$A=(A_{ij})_{\Lambda \times \Lambda}$ given by the elements
 \begin{align*}
  A_{ij} \stackrel{\rm def}{=}
   \begin{cases}
     \rho\varrho_{\ref{sg}} , & \mbox{if }i=j \\ 
     -  \frac{\varrho_{\ref{sg}}}{2}|M_{ij}|, & \mbox{if $i\neq j$}.
   \end{cases}
\end{align*}
and the vector~$\Phi=(\Phi_i)_{\Lambda}$ given by the elements
\begin{align*}
  \Phi_i \stackrel{\rm def}{=} \int |\nabla_i \varphi|^2 d\mu_\Lambda
\end{align*}
one can deduce from the estimates of Lemma~\ref{p_site_estimate} the inequality
\begin{align}\label{e_matrix_inequality_weak_inter}
	A \Phi
		\leq \sum_{k \in \supp f}e_k \int | \nabla f|^2 d\mu_\Lambda,
\end{align}
where~$e_{k}$ denotes the Euclidean basis vector for the site~$k$ and the inequality is understood component-wise. From the additional assumption~(\ref{e_ass_small_inter}), one directly sees that the matrix~$A$ is diagonally dominant.  This allows us to apply~\cite[Proposition~3.5]{bblMenz14} to obtain that $A$ is invertible, $A^{-1}$ has all non-negative entries, and that
\[ (A^{-1})_{ij} \lesssim \frac{1}{1+|i-j|^{2d+\alpha_{\ref{e_alg_decay}}}}.\] 
Therefore, we can apply the inverse~$A^{-1}$ to the inequality~\eqref{e_matrix_inequality_weak_inter} and obtain
\begin{align*}
  \Phi \leq A^{-1} \sum_{k \in \supp f}e_{k} \int | \nabla f|^2 d\mu_\Lambda .
\end{align*}
Because the last inequality has to be understood component-wise, one gets that for~$i \in \Lambda$
\begin{align}\label{e_tech_ingredient_HM_simple_case}  
 \int |\nabla_i \varphi|^2 d\mu_\Lambda
 	= \Phi_i
	& \leq  \sum_{k \in \supp f} (A^{-1})_{ik} \int | \nabla f|^2 d\mu_\Lambda \\
  	& \leq C \  \frac{1}{1+\dist (\supp f , i)^{2d+\alpha_{\ref{e_alg_decay}}}}  \ |\supp f| \int | \nabla f|^2 d\mu_\Lambda,
\end{align}
which yields the statement of Theorem~\ref{p_directional_poincare} under the additional assumption~\eqref{e_ass_small_inter} of weak interactions. 

\subsubsection*{The proof of the directional Poincar\'e inequality, Theorem~\ref{p_directional_poincare}}


However, in the general situation of Theorem~\ref{p_sg_dc}, the assumption~\eqref{e_ass_small_inter} is not satisfied. The interaction~$|M_{ij}|$ can be arbitrarily large for close sites~$i$ and~$j$. We attempt to circumvent this problem by considering blocks, where one expects that the interaction~$|M_{ij}|$ between close sites~$i$ and $j$ will not play an important role. Therefore instead of using Lemma~\ref{p_site_estimate}, we use a slightly different estimate as a starting point of our argument:

\begin{lem}\label{p_bl0_arbitrary_sets}
Let $\phi$ be given by the solution of~\eqref{e_weak_solution_elliptic}. For any subset~$S\subset \Lambda$ with~$\supp f \cap S = \emptyset$ it holds
\begin{equation}
\begin{split}
 &\varrho_{\ref{sg}}  \sum_{i \in S} 	\int |\nabla_i \varphi|^2 d\mu_\Lambda - \sum_{i \in S} \sum_{j\notin S} \frac{|M_{ij}|}{2} \left[ \int |\nabla_i \varphi|^2 d\mu_\Lambda + \int |\nabla_j \varphi|^2d\mu_\Lambda\right]
		\leq 0.\end{split}\label{e_block_estimate_general_S}
            \end{equation}
\end{lem}
We state the proof of Lemma~\ref{p_bl0_arbitrary_sets} in Section~\ref{ss_lemmas}. As a direct consequence of Lemma~\ref{p_bl0_arbitrary_sets} we get the following block estimate:
\begin{cor}\label{p_bl0}
Let $R\in \R_+$ and $\ell \in \Z^d$ such that $\supp f \cap B_R (\ell) = \emptyset $, where~$B_R(\ell)$ denotes the set
\begin{align*}
  B_{R} (\ell) = \left\{ i \in \Lambda \ | \ |i - \ell| \leq R  \right\}.
\end{align*}
  Then 
\begin{equation}
\begin{split}
 &\varrho_{\ref{sg}}  \sum_{i \in B_{R(\ell)}}
	\int |\nabla_i \varphi|^2 d\mu_\Lambda\\
	&- \sum_{i \in B_R(\ell)} \sum_{j\notin B_R(\ell)} \frac{|M_{ij}|}{2} \left[ \int |\nabla_i \varphi|^2 d\mu_\Lambda + \int |\nabla_j \varphi|^2d\mu_\Lambda\right]
		\leq 0.\end{split}\label{e_block_estimate}
            \end{equation}
\end{cor}

Fix a positive $R$ such that $R \notin \Z$ but $2R \in \Z$ to be determined later.  We choose $R$ in this way so that every point is in precisely one ball of radius $R$ centered on a point of the sub-lattice $2R\Z^d$. Now, we would like to apply the same linear algebra argument as outlined above on the level of blocks. Working on the sub-lattice $2R \Z^d$, one would define the vector~$\Phi=(\Phi_l)_{l\in2R \Z^d}$ according to
\begin{align*}
  \Phi_l = \sum_{i \in B_{R}(\ell)}
	\int |\nabla_i \varphi|^2 d\mu_\Lambda.
\end{align*}
Unfortunately, the estimate~\eqref{e_block_estimate} is still not well suited for this procedure: The corresponding matrix~$A$ on the block level would not be strictly diagonally dominant because the interaction~$M_{ij}$ between nearest neighbor blocks can be arbitrarily large. Even considering large blocks, i.e.~choosing $R \gg 1$ very large, would not help. \\

We circumvent this problem by an additional averaging process. Averaging the estimate~\eqref{e_block_estimate} over all choices~$\ell \in B_{R}(k)$ directly leads to the following statement.  
 
\begin{cor}\label{p_bl1}
Let $\phi$ be given by the solution of~\eqref{e_weak_solution_elliptic} and let $k\in 2R\Z^d$ be a site that satisfies~$\supp f \cap B_{2R}(k) = \emptyset$.  Then
\begin{equation}
  \label{e_bl1_1}
\begin{split}
 &\varrho_{\ref{sg}} \sum_{\ell \in B_{R}(k)} \sum_{i \in B_R(\ell)}
	\int |\nabla_i \varphi|^2 d\mu_\Lambda\\
	&- \sum_{\ell \in B_{R}(k)}\sum_{i \in B_R(\ell)} \sum_{j\notin B_R(\ell)} \frac{|M_{ij}|}{2} \left[ \int |\nabla_i \varphi|^2 d\mu_\Lambda + |\nabla_j \varphi|^2d\mu_\Lambda\right]
		\leq 0.\end{split}  
\end{equation}
\end{cor}

By the additional averaging of Corollary~\ref{p_bl1} we take advantage of the fact that the cardinality of the boundary is smaller than the cardinality of the bulk. As can be seen in Figure~\ref{f_average}, the boundary of the block becomes blurred. Hence, the effect of strong interactions between the boundary of blocks is getting averaged out. Another way to understand this proceeding is that any choice of block decomposition of the system is artificial and has nothing to do with the actual behavior of the spin system. The solution of this problem is to average over all possible choices of the block decomposition leading to an averaged estimate like the one in Corollary~\ref{e_block_estimate}.  \\


Let us have a closer look at the estimate~\eqref{e_bl1_1}. In each sum, terms appear multiple times. Therefore by changing the order of summation, we see that the estimate~\eqref{e_bl1_1} is equivalent to 
\begin{equation}\label{e_bl1}
\varrho_{\ref{sg}}\sum_{i \in B_{2R}(k)} p_i \int |\nabla_i \varphi|^2 d\mu_\Lambda
	- \sum_{i \in B_{2R}(k)} q_i \int |\nabla_i \varphi|^2 d\mu_\Lambda
	- \sum_{ i \in\Lambda} \kappa_{ik} \int |\nabla_i \varphi|^2 d\mu_\Lambda
 		\leq 0,
\end{equation}
where the weights~$p_i,q_i$ and~$\kappa_{ij}$ give the frequency with which each term appears in~\eqref{e_bl1_1}. One can check that these weights are given by
\begin{equation}\begin{split}\label{e_coefficients}
	p_i &\stackrel{\text{def}}{=} |\{j \in B_{R}(k): i \in B_R(j)\}| = |B_{R}(k)\cap B_R(i)|,\\
	q_i &\stackrel{\text{def}}{=} \sum_{\ell \in B_{R}(k)\cap B_R(i)} \sum_{j \notin B_R(\ell)} \frac{|M_{ij}|}{2}, ~~~\text{ and }\\
	\kappa_{ki} &\stackrel{\text{def}}{=} \sum_{\ell \in B_{R}(k)\cap B_R(i)^c} \sum_{j \in B_R(\ell)} \frac{|M_{ij}|}{2}.
\end{split}\end{equation}
Though $p_i$ and $q_i$ depend on $k$, this dependence is weak (see Lemma~\ref{p_coefficients} below).   As such, for notational simplicity, we do not adorn them with a $k$ subscript.\\

\begin{figure}[t]
\centering
\begin{overpic}[scale=.410]{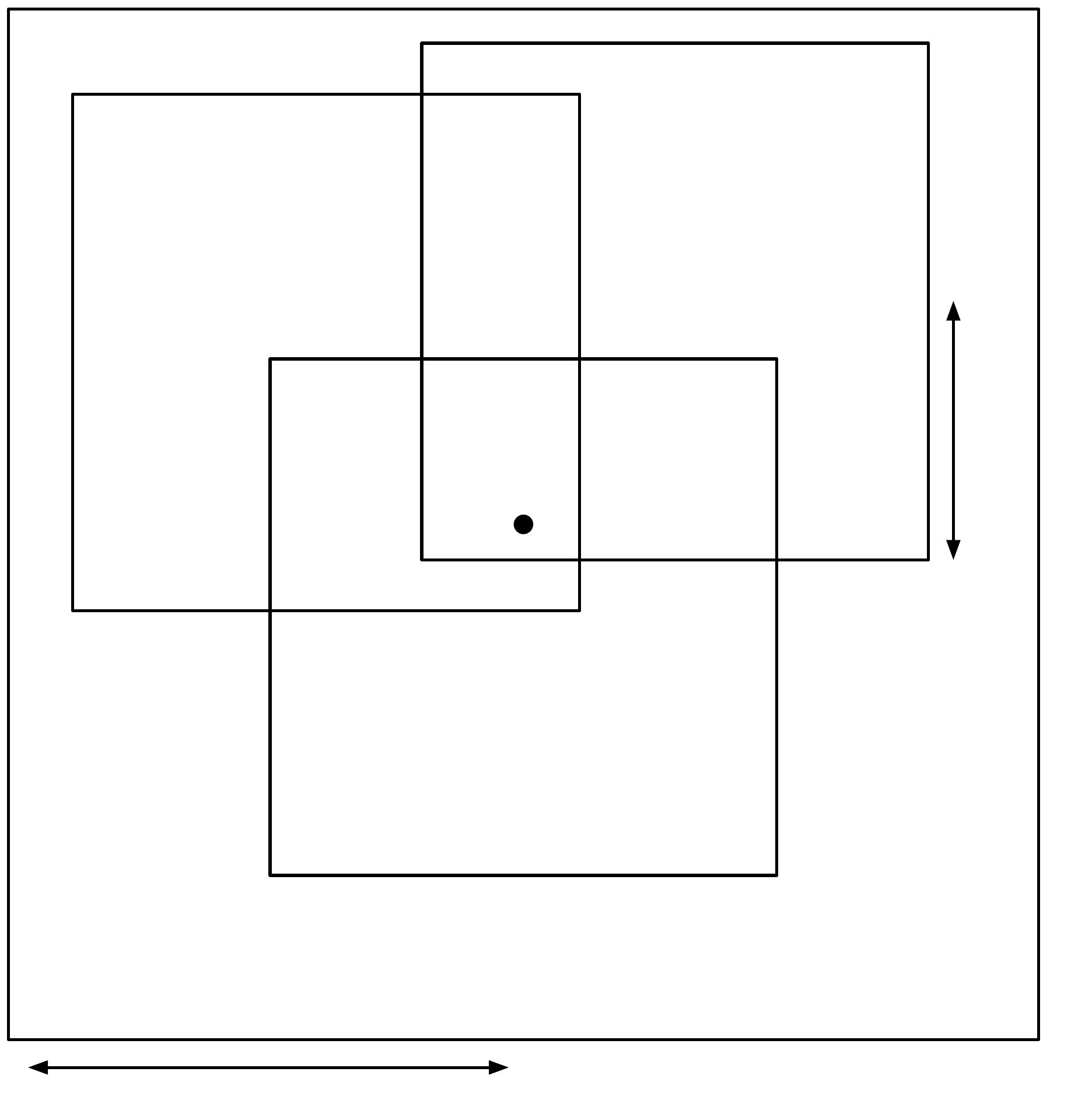}
\put(22,0){$2R$}
\put(86,65){$R$}
\put(43,55){$2Rj$}
\end{overpic}
~~~~~
\begin{overpic}[scale=.425]{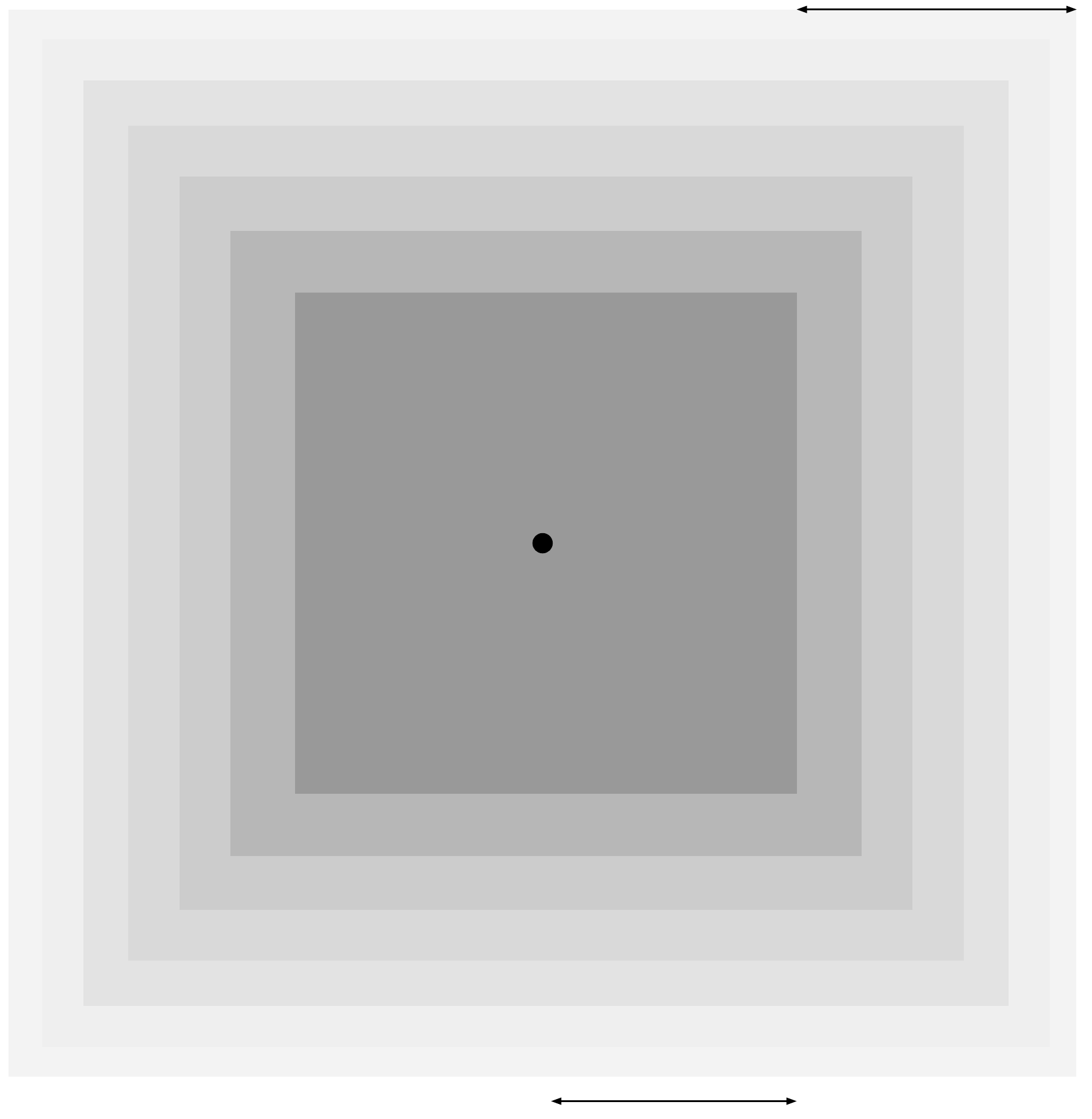}
\put(43,47){$2Rj$}
\put(60,4){$R$}
\put(85,100){$R$}
\end{overpic}
\caption{The figure on the left shows a sampling of the boxes being averaged over to create each term $\Phi_j$.  The shading in the figure on the right represents the number of times each point $i$ occurs in the averaging (represented below by $p_i$).}\label{f_average}
\end{figure}
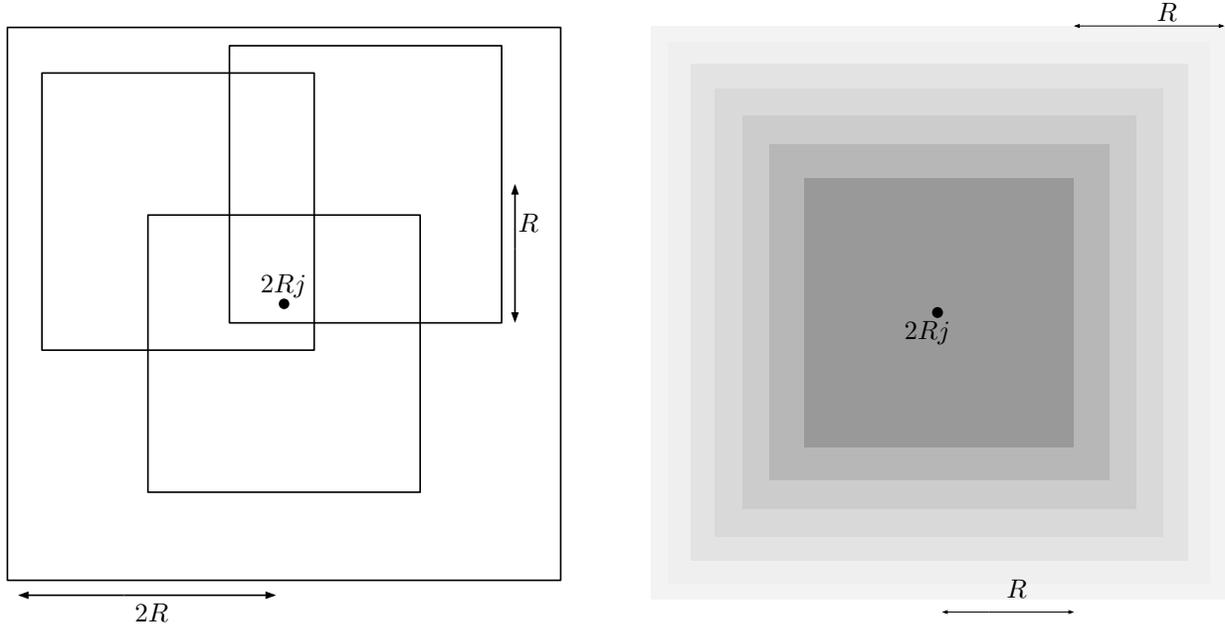

The crucial point in the proof of Theorem~\ref{p_directional_poincare} is to analyze the behavior of the weights when $R$ is large.  We summarize this analysis in the following lemma.  We defer the proof of this lemma until Section~\ref{ss_lemmas}.
%
%

\begin{lem}\label{p_coefficients}
Let $p_i$, $q_i$, and $\kappa_{kj}$ be defined as in~\eqref{e_coefficients}.  If $i \in B_{R}(k)$, then we have
\begin{equation}\label{e_p_estimate}
	p_i \gtrsim R^d
.\end{equation}
If $i \in B_{2R}(k)$, then we have
\[ q_i \lesssim R^{d-1}.\]
Finally, for any $\epsilon > 0$, we have
\begin{align}\label{e_estimate_block_interaction}
\kappa_{kj} \lesssim
			\min\left\{\frac{R^{d-\overline\alpha_{\ref{e_alg_decay}}+\epsilon}}{[1+\dist(j, B_{2R}(k))]^{d+\epsilon}},\frac{R^{2d}}{[1+\dist(j, B_{2R}(k))]^{2d+\alpha_{\ref{e_alg_decay}}}}\right\},
\end{align}
where we define
\[
	\overline\alpha_{\ref{e_alg_decay}} \stackrel{\text{def}}{=} \min\{\alpha_{\ref{e_alg_decay}},1\}.
\]
\end{lem}

The weight for the bulk, $(\rho_{\ref{sg}} p_i - q_i)$ when $i \in B_{R}(k)$, scales like $R^d$ whereas the other weights $\kappa_{ki}$, representing the interaction, scale like $R^{d-\overline\alpha_{\ref{e_alg_decay}}+\epsilon}$.  Therefore, for large $R$, one gets weak interaction analogous to~\eqref{e_BL_diagonal_dominance}.  This is the blurred boundary effect that we discussed above (cf.~Figure~\ref{f_average}).\\

Notice that bound~\eqref{e_p_estimate} does not hold uniformly for $i \in B_{2R}(k)\setminus B_{R}(k)$.  Thus, these terms will not help with the diagonal dominance that we require to use the argument outlined above.  As such, we drop those terms from~\eqref{e_bl1}.  From here, we now compute a simplified form of~\eqref{e_bl1} using the results of Lemma~\ref{p_coefficients}:
\[
\begin{split}
	0
		&\geq \varrho_{\ref{sg}}\sum_{i \in B_{R}(k)} p_i \int |\nabla_i \varphi|^2 d\mu_\Lambda
		- \sum_{i \in B_{2R}(k)} CR^{d-1} \int |\nabla_i \varphi|^2 d\mu_\Lambda
		- \sum_{i\notin B_{R}(k)} \kappa_{ki} |\nabla_i \varphi|^2 d\mu_\Lambda\\
		&\geq \varrho_{\ref{sg}}\sum_{i \in B_{R}(k)} (p_i- CR^{d-1}) \int |\nabla_i \varphi|^2 d\mu_\Lambda
		- \sum_{i\notin B_{R}(k)} \left(\kappa_{ki} + CR^{d-1} \I_{B_{2R}(k)}(i)\right) |\nabla_i \varphi|^2 d\mu_\Lambda,
\end{split}
\]
where $C$ is a universal constant.
Here we simply used the bounds on $q_i$ and $\kappa_{ki}$ in Lemma~\ref{p_coefficients} and split these sums into two parts: the terms in $B_{R}(k)$ and the terms in $B_{2R}(k)\setminus B_{R}(k)$.  
Using that $p_i\gtrsim R^d$ from Lemma~\ref{p_coefficients}, we may choose $R$ large enough that
\begin{equation}\label{e_lattice_ineq}
0	\geq \varrho_{\ref{sg}} \sum_{i \in B_{R}(k)} \frac{p_i}{2} \int |\nabla_i \varphi|^2 d\mu_\Lambda
		- \sum_{i\notin B_{R}(k)} \tilde\kappa_{ki} |\nabla_i \varphi|^2 d\mu_\Lambda,
\end{equation}
where, for ease, we have defined
\[
	\tilde\kappa_{ki} \stackrel{\text{def}}{=} \kappa_{ki} + CR^{d-1} \I_{B_{2R}(k)\setminus B_{R}(k)}(i)
.\]
Here $C$ is a universal constant independent of $R$.  Notice that $\tilde\kappa_{ki}$ satisfies the same upper bound as $\kappa_{ki}$,~\eqref{e_estimate_block_interaction}.  In other words,
\begin{equation}\label{e_tildekappa}
	\tilde\kappa_{kj} \lesssim
			\min\left\{\frac{R^{d-\overline\alpha_{\ref{e_alg_decay}}+\epsilon}}{[1+\dist(j, B_{2R}(k))]^{d+\epsilon}},\frac{R^{2d}}{[1+\dist(j, B_{2R}(k))]^{2d+\alpha_{\ref{e_alg_decay}}}}\right\}
.\end{equation}


With this inequality, we are ready to apply the same linear algebra argument as outlined above.  To make the argument simpler, we now work exclusively on the sub-lattice $2R\Z^d$ by looking at points on $\Z^d$ and scaling them by $2R$.  To this end, for $j, k \in \Z^d$ we define
\begin{equation}\label{e_block_defn}
\begin{split}
	\Phi_j &\stackrel{\text{def}}{=} \sum_{i \in B_{R}(2Rj)} \int |\nabla_i \varphi|^2 d\mu_\Lambda, ~~\text{ and }~~\\
	\overline{\kappa}_{kj} &\stackrel{\text{def}}{=} \max_{i \in B_{R}(2Rj)} \frac{\tilde\kappa_{(2Rk)i}}{R^d}
.\end{split}
\end{equation}
It follows from our estimates on $\tilde\kappa_{ij}$,~\eqref{e_tildekappa}, that
  \begin{align}
    	\overline\kappa_{kj} & \lesssim \frac{R^{-(\overline\alpha_{\ref{e_alg_decay}}-\epsilon)}}{(1 + \dist \left( B_{2R} (2Rk) , B_{2R}(2Rj) \right))^{d+ \epsilon}} \notag \\
    & \lesssim  \frac{R^{-(\overline \alpha_{\ref{e_alg_decay}}-\epsilon)}}{(1 + |k-j|)^{d+ \epsilon}}. \label{e_overlinekappa1}
  \end{align}
This is the estimate we will use to obtain diagonal dominance.  In addition, we have that
\begin{align}
	\overline\kappa_{kj} & \lesssim \frac{R^{d}}{(1 + \dist \left( B_{2R} (2Rk) , B_{2R}(2Rj) \right))^{2d+\alpha_{\ref{e_alg_decay}}}} \notag \\
		&\lesssim \frac{R^d}{(1 + |k-j|)^{2d+\alpha_{\ref{e_alg_decay}}}}. \label{e_overlinekappa2}
\end{align}
This is the estimate that we will use to obtain the correct rate of decay.\\

Now, we divide the equation~\eqref{e_lattice_ineq} by~$R^d$ and get, by using the estimate~\eqref{e_p_estimate} and the definitions in~\eqref{e_block_defn}, that
\begin{equation}\label{e_vector1}
	\varrho_{\ref{sg}} \Phi_k - C  \sum_{j\in \Z^d, j\neq k} \overline{\kappa}_{kj} \Phi_j \leq 0
,\end{equation}
for any point $k \in Z^d$ such that (see the assumptions of Corollary~\ref{p_bl1})
\begin{align*}
  \supp f \cap B_{2R}(2Rk) = \emptyset.
\end{align*}
 Here $C$ is a universal constant independent of $R$ and $k$.\\

Let us now consider the case~$k \in Z^d$ such that 
\begin{align*}
  \supp f \cap B_{2R}(2Rk) \neq \emptyset.
\end{align*}
In order to estimate~$\Phi_k$ in this situation, we cannot apply Corollary~\ref{p_bl1} and must proceed differently. In this situation, we may conclude from the Poincar\'e inequality, in the sense of Lemma~\ref{p_dual_poincare}, that
\[
	\Phi_k =  \sum_{i \in B_{R}(2Rj)} \int |\nabla_i \varphi|^2 d\mu_\Lambda
		\leq \sum_{i\in \Lambda} \int |\nabla_i \varphi|^2 d\mu_\Lambda
		\leq \frac{1}{\varrho_{\ref{sg}}^2} \int |\nabla f|^2 d\mu_\Lambda.
\]
The first inequality follows as $B_R \subset \Lambda$, and the second inequality is a result of Lemma \ref{p_dual_poincare}. This implies the inequality
\begin{equation}\label{e_vector2}
	\varrho_{\ref{sg}}^2 \Phi_{k} \lesssim \int |\nabla f|^2 d\mu_\Lambda
.\end{equation}

We define the set~$B_f$ according to
\begin{align}\label{e_def_M}
B_f \stackrel{\rm def}{=} \left\{ k \in \Z^d \ | \   \supp f \cap B_{2R}(2Rk) \neq \emptyset \right\}.
\end{align}

Because for any~$i$ the number~$\Phi_{k}\geq 0$, it follows from~\eqref{e_vector2} that for any~$k \in B_f$
\begin{equation}\label{e_vector3}
	\varrho_{\ref{sg}}^2 \Phi_k - C  \sum_{j\in \Z^d, j\neq k} \varrho_{\ref{sg}} \overline{\kappa}_{kj} \Phi_j \lesssim \int |\nabla f|^2 d\mu_\Lambda.
\end{equation}
Then by combining~\eqref{e_vector1} and~\eqref{e_vector3}, we obtain
\begin{equation}\label{e_vector_ineq}
	A \Phi \lesssim  \sum_{k \in B_f} e_{k} \int |\nabla f|^2 d\mu_\Lambda.
\end{equation}
Here, $e_j$ is the vector with zeros in all entries except in the $j$th entry where there is a one. The matrix $A$ is given by
\[
	A_{ij} \stackrel{\rm def}{=} \begin{cases}
		\varrho_{\ref{sg}}^2, ~~~&\text{ if } i = j,\\
		-C \varrho_{\ref{sg}} \overline{\kappa}_{ij}, ~~~&\text{ if } i \neq j
	,\end{cases}
\]
with $C$ is as in \eqref{e_vector1}.  This vector inequality is to be understood component-wise.\\

As a result of equation~\eqref{e_overlinekappa1}, we may fix $R$ large enough, independently of $f$ and $g$, that this matrix is diagonally dominant.  Since we have chosen $R$ as such, we may absorb any $R$-terms into the universal constant implied by the $\lesssim$-notation.  We do this from now on.\\

This allows us to apply~\cite[Proposition~3.5]{bblMenz14} to obtain that $A$ is invertible, $A^{-1}$ has all non-negative entries, and that, by~\eqref{e_overlinekappa2},
\[
	(A^{-1})_{ij} \lesssim \frac{1}{[1+|i-j|]^{2d+\alpha_{\ref{e_alg_decay}}}}.
\]
This, along with~\eqref{e_vector_ineq}, gives us
\[
	\Phi \lesssim A^{-1} \sum_{k \in B_f} e_{k}  \int |\nabla f|^2 d\mu_\Lambda
\]
and hence, for any~$l \in \Z^d$
\begin{align}
	\Phi_l  & \lesssim \sum_{k \in B_f} \left( A^{-1} \right)_{lk} \int |\nabla f|^2 d\mu_\Lambda \notag \\
	& \lesssim \sum_{k \in B_f} \frac{1}{\left( 1+|k - l|\right)^{2d+\alpha_{\ref{e_alg_decay}}}} \int |\nabla f|^2 d\mu_\Lambda \notag \\
	& \lesssim  \left(\max_{k \in B_f} \frac{|B_f|}{\left( 1+ |k-l|\right)^{2d+\alpha_{\ref{e_alg_decay}}}} \right)    \int |\nabla f|^2 d\mu_\Lambda \notag \\
	& \lesssim  \frac{|\supp f|}{\left( 1+ \dist (2Rl,\supp f)\right)^{2d+\alpha_{\ref{e_alg_decay}}}}  \   \int |\nabla f|^2 d\mu_\Lambda    \label{e_nabla_phi_estimate}.
\end{align}
Here, we used in the last step that by the definition~\eqref{e_def_M} of~$B_f$, it clearly holds that
  \begin{align*}
    |B_f| \leq |\supp f|.
  \end{align*}
In addition, we used, without verification, that
\begin{align}\label{e_estimate_B_f_terms}
	\max_{k \in B_f} \frac{1}{\left( 1+ |k-l|\right)^{2d+\alpha_{\ref{e_alg_decay}}}}
		\lesssim  \frac{1}{\left( 1+ \dist (2Rl,\supp f)\right)^{2d+\alpha_{\ref{e_alg_decay}}} }.
\end{align}
\mbox{}\\

We postpone the verification of~\eqref{e_estimate_B_f_terms} and show how to conclude the proof of Theorem~\ref{p_directional_poincare}. For any $~i \in \Lambda$, there is~$l \in \Z^d$ such that $i \in B_{R}(2Rl)$. Using now the estimate~\eqref{e_nabla_phi_estimate} yields
\begin{align}
\int |\nabla_i \varphi |^2 d \mu_{\Lambda} & \leq \sum_{k \in B_{R}(2Rl)} \int |\nabla_k \varphi |^2 d \mu_{\Lambda}  = \Phi_l \notag \\
& \lesssim \frac{|\supp f|}{\left( 1+ \dist (2Rl,\supp f)\right)^{2d+\alpha_{\ref{e_alg_decay}}} } \int |\nabla f|^2 d\mu_\Lambda \notag \\
& \lesssim \frac{|\supp f|}{\left( 1+ \dist (i,\supp f)\right)^{2d+\alpha_{\ref{e_alg_decay}}} }  \int |\nabla f|^2 d\mu_\Lambda,
 \label{e_final_estimation_nabla_i_phi}
\end{align}
where we used in the last step that
\begin{align}
	\frac{1}{\left( 1+ \dist (2Rl,\supp f)\right)^{2d+\alpha_{\ref{e_alg_decay}}} }
		\lesssim \frac{1}{\left( 1+ \dist (i,\supp f)\right)^{2d+\alpha_{\ref{e_alg_decay}}} } . \label{e_estiamte_center_block_arbitrary_block_element}
\end{align}

We want to note that~\eqref{e_final_estimation_nabla_i_phi} already contains the the statement of Theorem~\ref{p_directional_poincare}. Therefore in order to complete the proof of Theorem~\ref{p_directional_poincare} it is only left to verify the estimates~\eqref{e_estimate_B_f_terms} and~\eqref{e_estiamte_center_block_arbitrary_block_element}.\\

Let us now turn to the verification of the estimate~\eqref{e_estimate_B_f_terms}. It follows from the triangle inequality that for any element~$k \in B_f$ 
\begin{align*}
  \dist (2Rl,\supp f) &\leq |2Rl -2Rk| + \dist(2Rk, \supp f) \\
& \leq |2Rl-2Rk| + R , 
\end{align*}
where we used in the last step the definition~\eqref{e_def_M} of the set~$B_f$. This yields in particular that
\begin{align*}
  \dist (2Rl,\supp f) &\leq \min_{k \in B_f} |2Rl - 2Rk| + R .
\end{align*}
Therefore we get 
\begin{align*}
  \max_{k \in B_f} \frac{1}{\left( 1+ |k-l|\right)^{2d+\alpha_{\ref{e_alg_decay}}}} & \leq \max_{k \in B_f}  \frac{(2R)^{2d+\alpha_{\ref{e_alg_decay}}}}{(2R +  |2Rk-2Rl| )^{2d+\alpha_{\ref{e_alg_decay}}}} \\
&  \leq \frac{(2R)^{2d+\alpha_{\ref{e_alg_decay}}}}{(R +  \dist(2Rl, \supp f) )^{2d+\alpha_{\ref{e_alg_decay}}}} \\
& \leq \frac{(2R)^{2d+\alpha_{\ref{e_alg_decay}}}}{(1 +  \dist(2Rl, \supp f) )^{2d+\alpha_{\ref{e_alg_decay}}}},
\end{align*}
which is the desired estimate~\eqref{e_estimate_B_f_terms}.\\

Finally, it is only left to deduce the estimate~\eqref{e_estiamte_center_block_arbitrary_block_element}. Because~$i \in B_{R}(2Rl)$, we have by the triangle inequality that
\begin{align*}
  \dist(i, \supp f) \leq \dist(i, 2Rl) + \dist(2Rl, \supp f)  \leq R+ \dist (2Rl, \supp f).
\end{align*}
Hence we get
\begin{align*}
     \frac{1}{\left( 1+ \dist (2Rl,\supp f)\right)^{2d+\alpha_{\ref{e_alg_decay}}} } & \leq \frac{R^{2d+\alpha_{\ref{e_alg_decay}}}}{\left( 2R+ 2R\dist (2Rl,\supp f)\right)^{2d+\alpha_{\ref{e_alg_decay}}} }  \\
 & \leq \frac{(2R)^{2d+\alpha_{\ref{e_alg_decay}}}}{\left( R+ \dist (i,\supp f)\right)^{2d+\alpha_{\ref{e_alg_decay}}} }\\
& \leq \frac{(2R)^{2d+\alpha_{\ref{e_alg_decay}}}}{\left( 1+ \dist (i,\supp f)\right)^{2d+\alpha_{\ref{e_alg_decay}}}},
\end{align*}
which is the desired estimate~\eqref{e_estiamte_center_block_arbitrary_block_element} and completes the proof of Theorem~\ref{p_directional_poincare}.

\subsection{Proof of auxiliary Lemma~\ref{p_bl0_arbitrary_sets} and Lemma~\ref{p_coefficients}}\label{ss_lemmas}

In this section we provide the two missing ingredients of the proof of Theorem~\ref{p_directional_poincare}, namely Lemma~\ref{p_bl0_arbitrary_sets} and Lemma~\ref{p_coefficients}. We first prove Lemma~\ref{p_bl0_arbitrary_sets} by arguing similarly as in~\cite{bblMenz14}, though here we do so on arbitrary sets~$S\subset \Lambda$ with~$\supp f \cap S = \emptyset$ and we apply the Cauchy-Schwarz inequality in a different way.
\begin{proof}[Proof of Lemma~\ref{p_bl0_arbitrary_sets}:] 
Because the set $S\subset \Lambda$ satisfies~$\supp f \cap S = \emptyset$, it holds for any site~$j \in S$ that $\nabla_j f = 0$ and therefore
\begin{align*}
  0 = \int \nabla_j f \nabla_j \phi d\mu_\Lambda.
\end{align*}
Applying integration by parts, we obtain
\begin{equation}\label{e_ibp_ss}
\begin{split}
\int \nabla_j f \nabla_j \phi d\mu_\Lambda
	&= - \int \nabla_j\nabla_j \varphi \left( f - \int f d\mu_\Lambda\right) d\mu_\Lambda\\
		&~~~~+ \int \nabla_j \varphi \nabla_j H \left( f - \int f d\mu_\Lambda\right) d\mu_\Lambda
.\end{split}
\end{equation}
Now applying~\eqref{e_weak_solution_elliptic} to the terms on the right hand side of \eqref{e_ibp_ss} with $\xi = \nabla_j\phi\nabla_j \varphi$ and $\xi = \nabla_j \nabla_j H$, respectively, yields
\begin{equation}\label{e_ibp_ss_step2}
\begin{split}
\int \nabla_j f \nabla_j \varphi d\mu_\Lambda
	&= - \int \langle\nabla (\nabla_j\nabla_j \varphi) , \nabla \varphi\rangle d\mu_\Lambda
		+ \int \langle\nabla(\nabla_j \varphi \nabla_j H) , \nabla \varphi\rangle d\mu_\Lambda \\
	&= - \int \langle\nabla (\nabla_j\nabla_j \varphi) , \nabla \varphi\langle d\mu_\Lambda
		+ \int \nabla_j H \langle\nabla(\nabla_j \varphi) , \nabla \varphi\rangle d\mu_\Lambda \\
		&~~~ + \int \nabla_j \varphi \langle\nabla(\nabla_j H), \nabla \varphi\rangle d\mu_\Lambda \\
\end{split}
\end{equation}
Here, $\langle\cdot,\cdot\rangle$ is the scalar product.  We focus on the second term on the right hand side.  Here, using integration by parts we compute
\[\begin{split}
	\int \nabla_j H \langle\nabla(\nabla_j \varphi), \nabla \varphi \rangle d\mu_\Lambda
		&= - \frac{1}{Z} \int \langle\nabla(\nabla_j \varphi) , \nabla \varphi \rangle \nabla_j\left(\exp\{-H(x)\}\right) dx\\
		&= \int \langle\nabla\left(\nabla_j \nabla_j \varphi\right),\nabla \varphi \rangle d\mu_\Lambda
			+ \int \langle\nabla\left(\nabla_j \varphi\right) , \nabla\left( \nabla_j \varphi\right)\rangle d\mu_\Lambda
.\end{split}\]
Now, we insinuate this equality into \eqref{e_ibp_ss_step2} and sum over all elements~$j \in S$ to obtain
\begin{equation}\label{e_sumzero}
\begin{split}
0 &= \sum_{j\in S} \int \nabla_j f \nabla_j \varphi d\mu_\Lambda\\
	&= \sum_{j \in S}\int \langle\nabla\left(\nabla_j \varphi\right) , \nabla\left(\nabla_j \varphi\right)\langle d\mu_\Lambda + \sum_{j\in S}\int \nabla_j \varphi \langle\nabla\left(\nabla_j H\right) , \nabla \varphi \rangle d\mu_\Lambda\\
	&= \int \sum_{j\in S} \sum_k \Big[ |\nabla_k \nabla_j \varphi|^2 + \nabla_j \varphi \nabla_j \nabla_k H \nabla_k \varphi \Big] d\mu_{\Lambda}
.\end{split}
\end{equation}

In order to conclude, we need to borrow an inequality from~\cite[eqn. (1.9)]{bblLedoux}, which is that
\[0 \geq \varrho_{\ref{sg}} \sum_{j\in S}\int |\nabla_j\varphi|^2 d\mu_\Lambda -
	\int \left(\sum_{i,j\in S}|\nabla_j\nabla_i \varphi|^2 + \sum_{i,j\in S}\nabla_j\varphi M_{ij} \nabla_i \varphi\right) d\mu_\Lambda
.\]
We note that we applied the inequality from~\cite{bblLedoux} on the Gibbs measure $\mu_{S}$ and then integrated in the remaining variables $\Lambda\setminus S$ to obtain the above inequality.  Now, adding this to \eqref{e_sumzero} and using that $M_{ij} = \nabla_i \nabla_j H$, we obtain
\begin{equation}\label{e_averaged}
\begin{split}
0	\geq &\varrho_{\ref{sg}} \sum_{j\in S} \int |\nabla_j \varphi|^2 d\mu_\Lambda
	+ \int \sum_{i \in S, j\notin S} |\nabla_i\nabla_j \varphi|^2 d\mu_\Lambda\\
	&+ \int \sum_{i \in S, j\notin S} \nabla_i \varphi M_{ij} \nabla_j \varphi d\mu_\Lambda.
\end{split}
\end{equation}
To conclude, we simply apply Young's inequality, $2ab \leq a^2 + b^2$, to the last term on the right hand side.
\end{proof}

\begin{figure}[t]
\centering
	\begin{overpic}[scale=.65]{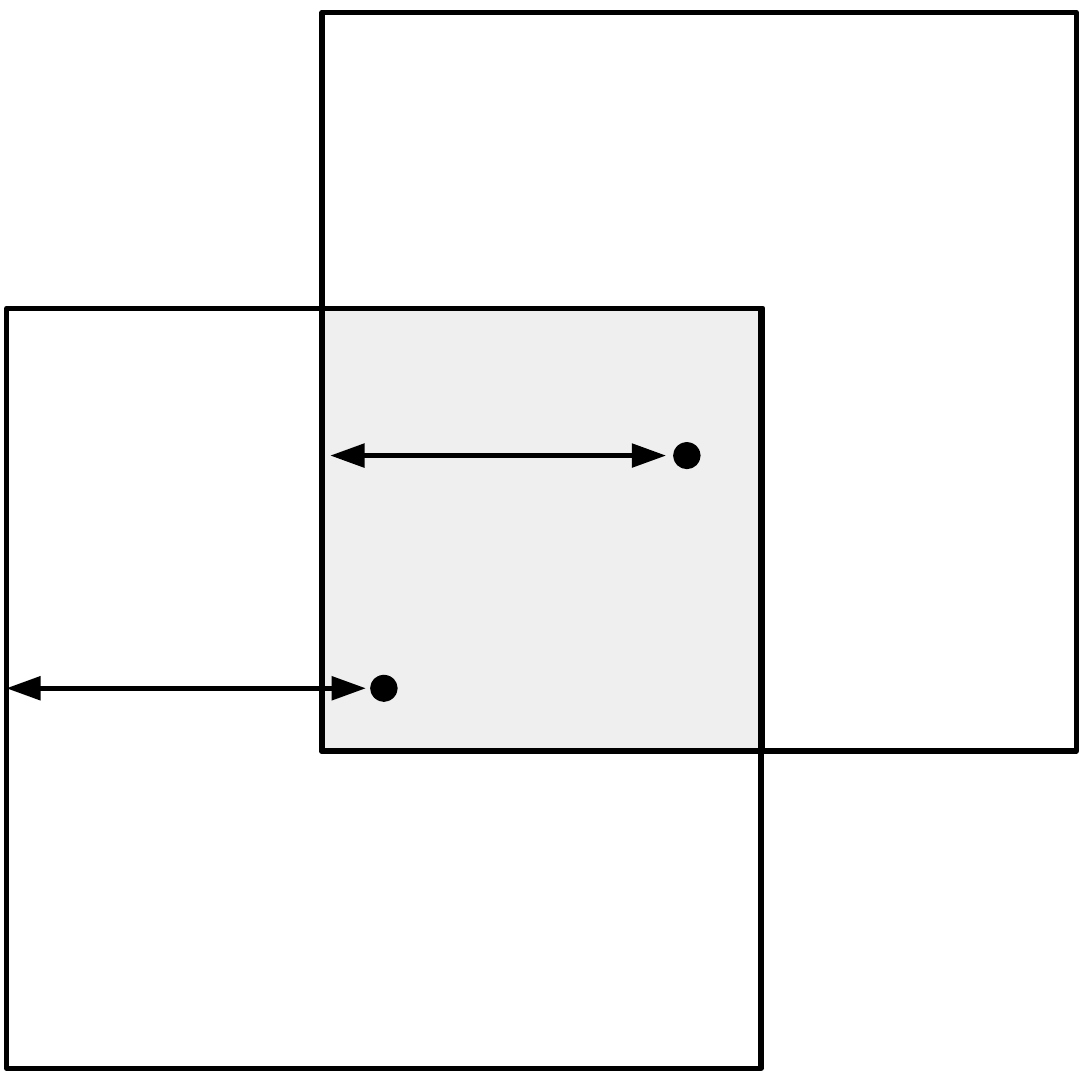}
		\put(44.5,60.5){$R$}
		\put(17,39.5){$R$}
		\put(62.5,60.5){$i$}
		\put(34,39.5){$k$}
	\end{overpic}
	\caption{The volume of the shaded region is $p_i$.  It is easy to see visually why $p_i$ is of the order $R^d$.}\label{f_p_i}
\end{figure}
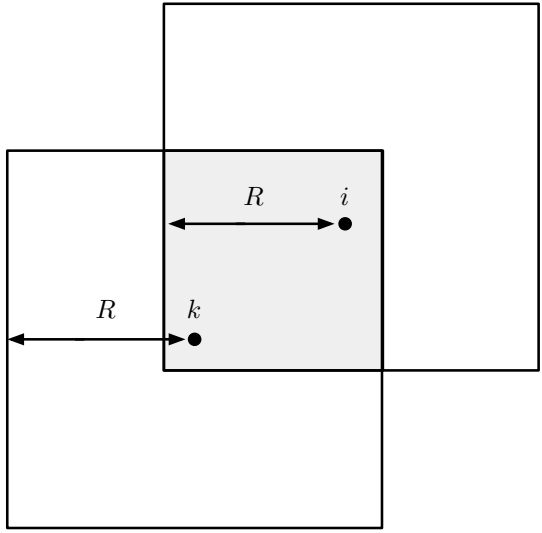

Finally, we prove Lemma~\ref{p_coefficients} by directly estimating the coefficients.
\begin{proof}[Proof of Lemma~\ref{p_coefficients}:] 
We first note that the lower bound on $p_i$ is a result of a simple counting argument, which is best illustrated by the volume of the shaded region in Figure~\ref{f_p_i}.  Indeed, one may assume by invariance and symmetry that $i = 0$ and that $k \in [0,R]^d$.  In this case, it is easy to see that $[0,R]^d \subset B_{R}(k)\cap B_R(i)$.  Taking the volume of each set finishes the proof.\\

We now approximate $q_i$.  The approach is a straightforward computation, taking advantage of re-writing the sums in polar coordinates.   To this end, we compute
\[\begin{split}
q_i
	&= \sum_{\ell\in B_{R}(k)\cap B_R(i)} \sum_{j\notin B_R(\ell)} \frac{|M_{ij}|}{2}\\
	&\lesssim \sum_{\ell\in B_{R}(k)\cap B_R(i)} \sum_{j\notin B_R(\ell)} [1+|i-j|]^{-(2d+\alpha_{\ref{e_alg_decay}})}\\
	&\leq \sum_{\ell \in B_{R}(k)\cap B_R(i)} \sum_{\{j : |j-i| \geq \dist(i,B_R(\ell)^c)\}} [1+|i-j|]^{-(2d+\alpha_{\ref{e_alg_decay}})}\\
	&\lesssim \sum_{\ell\in B_{R}(k)\cap B_R(i)} [1+\dist(i, B_R(\ell)^c)]^{-(d+\alpha_{\ref{e_alg_decay}})}\\
	&= \sum_{\ell\in B_{R}(k)\cap B_R(i)} [1+ R - |i-\ell|]^{-(d+\alpha_{\ref{e_alg_decay}})}
.\end{split}\]
Above we used the bound~\eqref{e_alg_decay} on $M_{ij}$ in the first inequality.  In the next inequality, since $|i - \ell|\leq R$, we used that if $j \notin B_R(\ell)$ then $|j - i| \geq \dist(i,B_R(\ell)^c)$.  Finally, we simply computed the inner sum.  Notice that $|i-\ell|\leq R$ in every term in the sum, and so we may change coordinates to obtain
\[
q_i
	\lesssim \sum_{m\in B_R(0)} [1 + R - |m|]^{-(d+\alpha_{\ref{e_alg_decay}})}
.\]
Recall that $R \notin \Z$ and $2R\in \Z$.  Switching to polar coordinates yields
\[\begin{split}
q_i	&\lesssim \sum_{r=0}^{R-1/2} \sum_{|m| = r} [1 + R - r]^{-(d+\alpha_{\ref{e_alg_decay}})}
	\lesssim \sum_{r=0}^{R-1/2} (1+r)^{d-1} [1+R-r]^{-(d+\alpha_{\ref{e_alg_decay}})}\\
	&\lesssim (1+R)^{d-1}  \sum_{r=0}^{R-1/2} [1+R-r]^{-(d+\alpha_{\ref{e_alg_decay}})}
	\lesssim R^{d-1}.
\end{split}\]
Here we bounded $r$ by $R$ and used the summability of $[1 + R - r]^{-(d+\alpha_{\ref{e_alg_decay}})}$.  This finishes the bound for $q_i$.

Now we need to bound $\kappa_{kj}$.  Recall from the statement of Lemma~\ref{p_coefficients}, equation~\eqref{e_estimate_block_interaction}, that we must prove two inequalities for $\kappa_{kj}$~\eqref{e_coefficients}.  We begin by showing that
\begin{equation}\label{e_kappa_1}
	\kappa_{kj} \lesssim \frac{R^{d-\overline\alpha_{\ref{e_alg_decay}} + \epsilon}}{[1 + \dist(j,B_{2R}(k))]^{d+\epsilon}}.
\end{equation}
First we handle the case when $j \in B_{2R}(k)$.  Here we compute that
\[\begin{split}
\kappa_{kj}
	&= \sum_{\ell\in B_{R}(k)\cap B_R(j)^c} \sum_{i \in B_R(\ell)} \frac{|M_{ij}|}{2}\\
	&\lesssim \sum_{\ell\in B_{R}(k) \cap B_R(j)^c}\sum_{i \in B_R(\ell)} [1 + |i-j|]^{-(2d+\alpha_{\ref{e_alg_decay}})}\\
	&\lesssim \sum_{\ell\in B_{R}(k)\cap B_R(j)^c} [1 + \dist(j, B_R(\ell))]^{-(d+\alpha_{\ref{e_alg_decay}})}\\
	&= \sum_{\ell\in B_{R}(k)\cap B_R(j)^c} [1 + |j - \ell| - R]^{-(d+\alpha_{\ref{e_alg_decay}})}.
\end{split}\]

We first use the decay of $M_{ij}$ given by equation \eqref{e_alg_decay}, then we computed the inner sum and rewrote $\dist(j,B_R(\ell))$ as $|j-\ell| - R$.  Noticing that $\ell \in B_{R}(k)\cap B_R(j)^c$ and that $j\in B_{2R}(k)$, we have that $R + 1/2 \leq |j-\ell|\leq 3R$.  Here, the lower bound is due to the fact that we have chosen $R \notin \Z$ such that $2R \in \Z$.  Hence, we may switch to polar coordinates to obtain
\[\begin{split}
\kappa_{kj}
	&\lesssim\sum_{r = R+1/2}^{3R} \sum_{|\ell - j|=r} [1 + r - R]^{-(d+\alpha_{\ref{e_alg_decay}})}\\
	&\lesssim \sum_{r=R+1/2}^{3R} r^{d-1}[1 + r - R]^{-(d+\alpha_{\ref{e_alg_decay}})} \\
	&\lesssim R^{d-1}
,\end{split}\]
where the last step follows by bounding $r^{d-1}$ by $(3R)^{d-1}$ and using that $[1 + r - R]^{-(d+\alpha_{\ref{e_alg_decay}})}$ is summable.  We note that for  $j \in B_{2R}(k)$ it holds that
\begin{align*}
1=  1 + \dist(j, B_{2R}(k)) .
\end{align*}
Therefore, by recalling that $\overline \alpha_{\ref{e_alg_decay}} = \min\{\alpha_{\ref{e_alg_decay}}, 1\}$, we obtain that
\[
	\kappa_{kj}
		\lesssim R^{d-1}
		\lesssim \frac{R^{d-\overline\alpha_{\ref{e_alg_decay}}+\epsilon}}{[1 + \dist(j, B_{2R}(k))]^{d+\epsilon}},
\]
finishing the proof when $j \in B_{2R}(k)$.\\

To finish the proof of \eqref{e_kappa_1}, we need to bound $\kappa_{kj}$ in the case that $j \in B_{2R}(k)^c$.  We compute that
\[\begin{split}
\kappa_{kj}
	&= \sum_{\ell \in B_{R}(k)} \sum_{i \in B_R(\ell)} \frac{|M_{ij}|}{2}\\
	&\lesssim \sum_{\ell\in B_{R}(k)} \sum_{i \in B_R(\ell)} [1+ |i-j|]^{-(2d+\alpha_{\ref{e_alg_decay}})}\\
	&\lesssim \sum_{\ell \in B_{R}(k)} [1 + \dist(j, B_R(\ell))]^{-(d+\alpha_{\ref{e_alg_decay}})}
.\end{split}\]
As before, here we simply used the bound on $M_{ij}$, \eqref{e_alg_decay}, and we estimated the inner sum.  Next, we pull the desired decay out of this sum, using that
\[
	\dist(j,B_{2R}(k)) \leq \dist(j, B_R(\ell))
\]
since $B_R(\ell) \subset B_{2R}(k)$ when $\ell \in B_{R}(k)$.  Hence, we obtain
\begin{equation}\label{e_kappa}
\kappa_{kj}
	\lesssim [1 + \dist(j, B_{2R}(k))]^{-(d+\epsilon)} \sum_{\ell \in B_{R}(k)} [1 + \dist(j, B_R(\ell))]^{-\alpha_{\ref{e_alg_decay}}+\epsilon}
.\end{equation}
Define $D =\dist(j, B_{2R}(k))$, and using polar coordinates again, we see that the sum in \eqref{e_kappa} is bounded as
\[\begin{split}
\sum_{\ell\in B_{R}(k)}[1 + \dist(j, B_R(\ell))]^{-\alpha_{\ref{e_alg_decay}}+\epsilon}
	&= \sum_{\ell\in B_{R}(k)}[1 + |j-\ell| - R]^{-\alpha_{\ref{e_alg_decay}}+\epsilon}\\
	&\lesssim \sum_{r= R+D +1/2}^{3R+D+1/2}\sum_{\substack{|\ell - j|=r,\\ \ell \in B_{R}(k)}} [1 + r - R]^{-\alpha_{\ref{e_alg_decay}}+\epsilon}\\
	&\lesssim \sum_{r = R+D+1/2}^{3R+D+1/2}  R^{d-1}[1 + r - R]^{-\alpha_{\ref{e_alg_decay}}+\epsilon}\\
	&\lesssim R^{d-1} R^{1 - \alpha_{\ref{e_alg_decay}} + \epsilon}
	= R^{d-\alpha_{\ref{e_alg_decay}}+\epsilon}
.\end{split}\]
In the second to last inequality, we used that there are at most $CR^{d-1}$ elements in $B_{2R}(k)$ of a fixed distance from $j$, where $C$ is a constant depending only on the dimension.  Combining this with~\eqref{e_kappa}, we arrive at
\[
	\kappa_{kj} \lesssim \frac{R^{d-\overline\alpha_{\ref{e_alg_decay}}+\epsilon}}{[1 + \dist(j, B_{2R}(k))]^{d+\epsilon}},
\]
as claimed.\\

Finally, to finish the proof of the lemma, we have to show that
\begin{align*}
\kappa_{kj} \lesssim  \frac{R^{2d}}{[1+\dist(j, B_{2R}(k))]^{2d+\alpha_{\ref{e_alg_decay}}}}.
\end{align*}
However, this estimate follows easily from the definition of~$\kappa_{kj}$. Indeed, we simply obtain the denominator in the bound for $\kappa_{kj}$ directly from the bound on $|M_{ij}|$, and we obtain the $R^{2d}$ term since there are $R^{2d}$ terms in the summation.

\end{proof}

%
%

\section{Proof of Theorem~\ref{p_bootstrap}: Slow decay of correlations implies fast decay of correlations}\label{s_bootstrap}

In this section we prove Theorem~\ref{p_bootstrap}, i.e.~that we can bootstrap a weaker decay of correlations to the same order of decay of the interaction terms if the Hamiltonian~$H$ satisfies the condition~\eqref{e_Lebowitz_conditions}. Keeping track of the constants from line to line can be a bit tedious so we use the following notation, similarly as in Section~\ref{s_sg_dc}.
\begin{rem}[Notation]
We write $a\lesssim b$ if and only if $a \leq C b$ for some constant $C > 0$ which depends only on the parameters of the Hamiltonian given in \eqref{e_alg_decay} and \eqref{e_diag_dominant}, on the dimension, and on the constants in \eqref{FWDC}.\\

 In addition, for convenience, with any set $A \subset \Lambda$, we write~$n \notin A$ to indicate that~$n \in \Lambda \backslash A$.
\end{rem}

Our iterative argument needs two ingredients. The first one is that for ferromagnetic interaction spin-spin correlations are always non-negative.  This is due to the following well known fact (see~eg.~\cite{Syl76} or~\cite[Lemma 2.1]{bblMenzNittka}).
\begin{lem}
  Assume that the interaction is ferromagnetic in the sense of~\eqref{e_ferromagnetic}. Then 
\begin{align}\label{e_ferro_non_neg_corr}
  \cov_{\mu_{\Lambda}} (x_i, x_j) \geq 0 \qquad \mbox{for any site~$i \in \Lambda$ and~$j \in \Lambda$.}
\end{align}
\end{lem}
The second one ingredient is more fundamental (see e.g.~\cite[Theorem 3.1]{Simon} or~\cite[Proposition 5.3]{bblMenzNittka}). 
\begin{lem}\label{e_key_lemma_bootstrapping}
  We assume that the Hamiltonian $H: \mathds{R}^\Lambda \to \mathds{R}$ given by~\eqref{d_Hamiltonian} satisfies the condition~\eqref{e_Lebowitz_conditions} i.e.
\begin{align*}
 \psi_i(z)=\psi_i(-z), \quad  s_i =0, \quad \mbox{and} \quad M_{ij} \leq 0 \qquad \mbox{for all sites $i \neq j$}.
\end{align*}
Let~$A \subset \Lambda$.  Then for any~$i,j \in \Lambda$ satisfying~$i \in A$ and~$j \notin A$ it holds
  \begin{align}
\cov_{\mu_\Lambda} (x_i, x_j) \leq \sum_{ \substack{k \in A  \\ n \notin A } } |M_{kn}|  \Big[ \cov_{\Lambda} (x_i, x_k) \cov_{\Lambda} (x_n, x_j) + \cov_{\Lambda} (x_i, x_n) \cov_{\Lambda}(x_k , x_j) \Big].
\label{e_lebowitz_consequence}    
  \end{align}
\end{lem}
For the short proof of Lemma~\ref{e_key_lemma_bootstrapping} we refer the reader to~\cite[Theorem 3.1]{Simon}. The argument is based on two inequalities: the Lebowitz inequality and the second GKS inequality. Our condition~\eqref{e_Lebowitz_conditions} on the Hamiltonian~$H$ ensures that those two inequalities are satisfied. Our iterative argument only needs the estimates~\eqref{e_ferro_non_neg_corr} and~\eqref{e_lebowitz_consequence}. As a consequence, the condition~\eqref{e_Lebowitz_conditions} is not used anymore in this section.
\\


Now, we are prepared to start the proof of Theorem~\ref{p_bootstrap}. For two fixed sites $i \in \Lambda$ and $j\in \Lambda$, we define the constant
\begin{equation}
	l =  (d+\alpha_{\ref{e_alg_decay_f}})\log_2(1 + |i-j|)  \label{d_constant_l}
\end{equation}
We fix a constant $L>0$, to be determined later, and assume without loss of generality that $|i-j| > 3(l+1)L$.  Using~\eqref{e_lebowitz_consequence} we obtain
\[\begin{split}
	\cov_{\mu_\Lambda}(x_i,x_j)
	&\leq \sum_{\substack{k\in B_{L}(i)\\n\notin B_{L}(i)}} |M_{kn}|
		\big[\cov_{\mu_\Lambda}(x_i,x_k) \cov_{\mu_\Lambda}(x_n,x_j)\\
		&~~~~ + \cov_{\mu_\Lambda}(x_i,x_n) \cov_{\mu_\Lambda}(x_k,x_j) \big]\\
	&=\sum_{\substack{k\in B_{L}(i)\\ n\notin B_{L}(i)}} |M_{kn}|\cov_{\mu_\Lambda}(x_i,x_n) \cov_{\mu_\Lambda}(x_k,x_j) \\
	 &~~~~+ \sum_{\substack{k\in B_{L}(i)\\n\in B_{L}(i)^c\cap B_{L}(j)^c}}|M_{kn}|\cov_{\mu_\Lambda}(x_i,x_k) \cov_{\mu_\Lambda}(x_n,x_j)\\
		&~~~~+ \underbrace{\sum_{\substack{k\in B_{L}(i)\\n\in B_{L}(i)^c\cap B_{L}(j)}} |M_{kn}| \cov_{\mu_\Lambda}(x_i,x_k) \cov_{\mu_\Lambda}(x_n,x_j)}_{\stackrel{\text{def}}{=}R_{ij}}.
\end{split}\]
Call the sum on the bottom line $R_{ij}$, as indicated above.  This is a ``remainder'' term which decays as fast as we would like.  We show the decay of $R_{ij}$ now.  By our choice of $i$, $j$, and $L$, we know that $|M_{kn}| \lesssim (1+|i-j|)^{-d-\alpha_{\ref{e_alg_decay_f}}}$.  Hence we may bound this last term as
\[\begin{split}
\sum_{\substack{k\in B_{L}(i)\\n\in B_{L}(i)^c\cap B_{L}(j)}}  &|M_{kn}|\cov_{\mu_\Lambda}(x_i,x_k) \cov_{\mu_\Lambda}(x_n,x_j)\\
	&\lesssim \sum_{\substack{k\in B_{L}(i)\\n\in B_{L}(i)^c\cap B_{L}(j)}} \frac{1}{(1+|i - j|)^{d+\alpha_{\ref{e_alg_decay_f}}}} \lesssim \frac{L^{2d}}{(1+|i - j|)^{d+\alpha_{\ref{e_alg_decay_f}}}}
.\end{split}\]
However, $L$ is a fixed constant independent of $i$ and $j$ (to be determined later), so we may consider $L^{2d}$ to be constant absorbed into the $\lesssim$ notation.  Hence this term is of the correct order.\\

Now we focus on proving the same decay for the other two terms.  To make the upcoming iterative step clearer, we first define
\begin{equation}\label{e_def_J_ik}
J_{ik} = \begin{cases} \sum_{n\notin B_{L}(i)} |M_{kn}| \cov_{\mu_\Lambda}(x_i, x_n), &\text{ for } k \in B_{L}(i)\\ \sum_{n\in B_{L}(i)} |M_{kn}| \cov_{\mu_\Lambda}(x_i, x_n), &\text{ for } k \notin B_{L}(i)  \end{cases}
\end{equation}
Now our inequality from above can be written as
\[
	\cov_{\mu_\Lambda}(x_i,x_j) \leq \sum_{k\notin B_{L}(j)} J_{ik} \cov_{\mu_\Lambda}(x_k,x_j) + R_{ij}.
\]
As we will see later, $J_{ik}$ is small; however, it is not of the order $|i-j|^{-(2d+\alpha_{\ref{e_alg_decay_f}})}$.  Hence we use the Lebowitz inequality to split the term $\cov_{\mu_\Lambda}(x_k,x_j)$ to obtain
\[\begin{split}
	\cov_{\mu_\Lambda}(x_i,x_j)
		&\leq \sum_{k\notin B_L(j)} J_{ik} \left[ \sum_{l\notin B_L(j)} J_{kl} \cov_{\mu_\Lambda}(x_l, x_j) + R_{kj} \right] + R_{ij}\\
		&= \sum_{k\notin B_L(j)}\sum_{l\notin B_L(j)} J_{ik}J_{kl} \cov_{\mu_\Lambda}(x_l, x_j) + \sum_{k\notin B_L(j)}\sum_{l\notin B_L(j)} J_{ik}R_{kj} + R_{ij}.
\end{split}\]
This is the step we iterate as in~\cite{bblMenzNittka}.  The reason we do so is the following.  We have established that $R_{ij}$ has the correct decay.  In addition, in the second to last term above, either $|i-k|$ or $|k-j|$ must be comparable to $|i-j|$.  Hence $J_{ij}$ and $R_{kj}$, through the $M$ terms in their definition, will decay like $|i-j|^{-d-\alpha_{\ref{e_alg_decay_f}}}$.  On the other hand, the first term need not decay of this same order.  However, we show that the $J$ terms are small; so iterating yields a product of many small terms.  This decays exponentially in the number of terms in the product and so we need only iterate $O(\log(|i-j|))$ times to get the desired decay. 
\\

Before continuing, we point out that the $R$ terms satisfy the following general bound.  If $|p - q| > 3L$ then
\begin{equation}\label{eq:R_bound}
	R_{pq}
		\lesssim \frac{1}{(1 + |p - q|)^{d + \alpha_{\ref{e_alg_decay_f}}}}.
\end{equation}

After $l$ iterations, we obtain that
\begin{equation}\label{iteration_scheme}
\begin{split}
	|\cov_{\mu_\Lambda}(x_i,x_j)|
		&\leq \sum_{k_1,\dots,k_l \notin B_{L}(j)} J_{ik_1}J_{k_1k_2}\cdots J_{k_{l-1}k_l} \cov_{\mu_\Lambda}(x_{k_l}, x_j)\\
		& + \sum_{k_1,\dots,k_{l-1} \notin B_{L}(j)} J_{ik_1}J_{k_1k_2}\cdots J_{k_{l-2}k_{l-1}} R_{k_{l-1}j}\\
		& + \sum_{k_1,\dots,k_{l-2} \notin B_{L}(j)} J_{ik_1}J_{k_1k_2}\cdots J_{k_{l-3}k_{l-2}} R_{k_{l-2}j} \\
		&+ \cdots + R_{ij}
.\end{split}\end{equation}
We show in the sequel that the first term is exponentially decaying in $l$, while the other terms each have a component of the order $|i-j|^{-(d+\alpha_{\ref{e_alg_decay_f}})}$.\\


To obtain the exponential decay mentioned above, we show that
\begin{equation}\label{eq:J_small}
	\sum_k J_{mk} \leq 1/2
		~~~\text{ and }~~~
		\sum_m J_{mk} \leq 1/2.
\end{equation}
Before proving this, we show how to finish the proof of Theorem~\ref{p_bootstrap} assuming these two inequalities.
\\

Assuming~\eqref{eq:J_small}, we obtain that the first term in~\eqref{iteration_scheme} is bounded as
\begin{equation}\label{eq:first_term}
	\sum_{k_1,\dots, k_l \notin B_L(j)} J_{ik_1}\cdots J_{k_{l-1} k_l} \cov_{\mu_\Lambda}(x_{k_l},x_j)
		\lesssim \left(\frac{1}{2}\right)^l.
\end{equation}
Here we used our bound \eqref{FWDC} to bound the covariance above by a uniform constant.  As we see below, this and our choice of $l$ implies the correct order of decay for this term.  As such, for now, we turn to the other terms in the sum~\eqref{iteration_scheme}\\

Hence, to finish we need only estimate the other terms in~\eqref{iteration_scheme}.  First, for a fixed natural number $r \leq l$, we must bound sums of the form
\[
	\sum_{k_1,\dots, k_{r}\notin B_{L}(j)} J_{k_0k_1}J_{k_1k_2}\cdots J_{k_{r-1}k_{r}}R_{k_rk_{r+1}}.
\]
where we define $k_0 = i$ and $k_{r+1} = j$ for notational convenience.  Since $k_0, k_1, \dots, k_{r+1}$ is a sequence of points connecting $i$ and $j$, we may find $s \in \{1,2,\dots,r+1\}$ such that $|i - j|/ (r+1) < |k_s - k_{s-1}|$.  Define
\[
	A_{s,r}
		\stackrel{\text{def}}{=}
			\left\{(k_1,\dots,k_r) \notin B_L(j)^r: \frac{|i-j|}{r+1} < |k_s-k_{s-1}|  \right\}.
\]
We may then re-write and bound our sum as
\begin{equation}\label{eq:asr}
	\sum_{k_1,\dots, k_{r}\notin B_{L}(j)} J_{k_0k_1}J_{k_1k_2}\cdots J_{k_{r-1}k_{r}}R_{k_rj}
		\lesssim \sum_{s=1}^{r+1}\sum_{A_{s,r}}  J_{k_0k_1}J_{k_1k_2}\cdots J_{k_{r-1}k_{r}}R_{k_rj}.
\end{equation}
For notational simplicity, we only show how to bound the terms for $s \leq r$; however, the bound for $s = r+1$ is obtained in the exact same manner.  By the definition of $A_{s,r}$, by $r\leq l$, and the fact that $|i-j| >3(l+1)L$ we have that
\[
	|k_s - k_{s-1}| \geq \frac{|i-j|}{r+1} \geq \frac{|i-j|}{l+1} \geq 3L,
\]
and hence that
\[
	|k_s - k_{s-1}| - L \geq |k_s - k_{s-1}| - \frac{1}{3}|k_s - k_{s-1}|
		=  \frac{2}{3}|k_s - k_{s-1}|
		\gtrsim \frac{|i-j|}{r}
		\gtrsim \frac{1 + |i-j|}{r}.
\]
In view of the definition~\eqref{e_def_J_ik} of $J_{k_{s-1}k_s}$ and the inequalities above, the fact that $(k_1,\dots, k_r)\in A_{s,r}$ implies that 
\[\begin{split}
	J_{k_{s-1}k_s}
		\lesssim \frac{1}{(1 + |k_s - k_{s-1}| - L)^{d+\alpha_{\ref{e_alg_decay_f}}}}
		\lesssim \frac{r^{d+\alpha_{\ref{e_alg_decay_f}}}}{(1 + |i-j|)^{d+\alpha_{\ref{e_alg_decay_f}}}}.
\end{split}\]

Using this in~\eqref{eq:asr}, we obtain
\begin{equation}\label{eq:iteration_bound}
\begin{split}
	\sum_{k_1,\dots, k_{r}\notin B_{L}(j)} &J_{k_0k_1}J_{k_1k_2}\cdots J_{k_{r-1}k_{r}}R_{k_rj}\\
		&\lesssim \sum_{s=1}^{r+1} \sum_{(k_1,\dots,k_r)\in A_{s,r}} J_{k_0k_1} \cdots J_{k_{s-2}k_{s-1}}\left(\frac{r^{d+\alpha_{\ref{e_alg_decay_f}}}}{(1 + |i-j|)^{d+\alpha_{\ref{e_alg_decay_f}}}}\right) J_{k_s k_{s+1}} \cdots J_{k_{r-1}k_r} R_{k_r j}.
\end{split}
\end{equation}
Applying the bound~\eqref{eq:J_small} to each of the $r-1$ remaining $J$ terms yields
\[
	\sum_{k_1,\dots, k_{r}\notin B_{L}(j)} J_{k_0k_1}J_{k_1k_2}\cdots J_{k_{r-1}k_{r}}R_{k_rj}
		\lesssim \sum_{s=1}^{r+1} \frac{r^{d+\alpha_{\ref{e_alg_decay_f}}}}{(1 + |i-j|)^{d+\alpha_{\ref{e_alg_decay_f}}}} \left(\frac{1}{2}\right)^{r-1}
		= \frac{r^{d+\alpha_{\ref{e_alg_decay_f}}}}{(1 + |i-j|)^{d+\alpha_{\ref{e_alg_decay_f}}}} \frac{r+1}{2^{r-1}}.
\]
To be explicit, in~\eqref{eq:iteration_bound} we first sum over $k_{s-1}$ using that $\sum_{k_{s-1}} J_{k_{s-2}k_{s-1}} < 1/2$, and then we sum over $k_{s-2}$ using~\eqref{eq:iteration_bound} again.  We continue in this manner to sum over all of $k_1, \dots, k_{s-1}$.  Then we sum over $k_s$ using that $\sum_{k_s} J_{k_sk_{s+1}} < 1/2$.  We continue in this manner to sum over all of $k_s, \dots, k_{r-1}$.  Each of these sums contributes a factor of $1/2$ above.  Finally, we use the summability over $k_r$ of $R_{k_r j}$ coming from the bound~\eqref{eq:R_bound} to conclude above.  \\

Plugging this bound and our bound of the first term~\eqref{eq:first_term} into~\eqref{iteration_scheme} yields
\[\begin{split}
	\cov_{\mu_\Lambda}(x_i,x_j)
		&\lesssim 2^{-l} + \frac{1}{(1+|i-j|)^{d+\alpha_{\ref{e_alg_decay_f}}}}\sum_{r=0}^{l-1} \frac{(r+1)r^{d+\alpha_{\ref{e_alg_decay_f}}}}{2^{r-1}}\\
		&\lesssim 2^{-l} + \frac{1}{(1+|i-j|)^{d+\alpha_{\ref{e_alg_decay_f}}}}.
\end{split}\]
Above, we used the summability of $r^{1+d+\alpha_{\ref{e_alg_decay_f}}} 2^{-r}$.  Recalling that by~\eqref{d_constant_l}
\[
	l = (d+\alpha_{\ref{e_alg_decay_f}}) \log_2(1 + |i-j|)
\]
gives us the desired inequality.
%
%
%
%
%
%
%
%
%
%
%
\\

All that remains is to show that~\eqref{eq:J_small} holds.  The proofs of the two bounds above are similar so we show only the estimate of $\sum_k J_{mk}$.  To do this we split this term up a few times.  We write
\begin{equation}\label{JSum}
\sum_k J_{mk} = \sum_{k\in B_{L}(m)} J_{mk} + \sum_{k\notin B_{L}(m)} J_{mk}
.\end{equation}
Let us focus on the first term.  Then
\begin{equation}\label{JFirstTerm}
\begin{split}
	\sum_{k\in B_{L}(m)} J_{mk}
		&= \sum_{k\in B_{L}(m)} \sum_{n\notin B_{L}(m)} |M_{kn}| \cov_{\mu_\Lambda}(x_m, x_n)\\
		&\lesssim \sum_{k\in B_{L}(m)} \sum_{n\notin B_{L}(m)} \frac{1}{(1 + |k-n|)^{d+\alpha_{\ref{e_alg_decay_f}}}} \frac{1}{L^{d + \alpha_{\ref{FWDC}}}}\\
		&\lesssim \frac{L^{d-\min\{1,\alpha_{\ref{e_alg_decay_f}}\}}}{L^{d  + \alpha_{\ref{FWDC}}}} \leq L^{ - \alpha_{\ref{FWDC}} - \min\{1,\alpha_{\ref{e_alg_decay_f}}\}}
.\end{split}
\end{equation}
Here we used the bound in~\eqref{FWDC} to bound $\cov_{\mu_\Lambda}(x_n, x_m)$, and we used that if $\alpha \in (0,1]$ and $R >0$ then
\[
	\sum_{n\in B_R}\sum_{m \notin B_{R}} \frac{1}{(|n-m| + 1)^{d+\alpha}} \lesssim R^{d-\alpha}.
\]
This inequality is computed similarly as those appearing in Section~\ref{ss_lemmas}. 
Hence,~\eqref{JFirstTerm} implies that the first term in \eqref{JSum} tends to zero as $L$ tends to infinity.  
As such, we may choose $L$ large enough such that the sum in \eqref{JFirstTerm} is less than 1/4.\\

Now we deal with the other sum.
\begin{equation}\label{JSecTerm}
\begin{split}
\left|\sum_{k\notin B_{L}(m)} J_{mk}\right|
	&= \sum_{\substack{k\notin B_{L}(m)\\n\in B_{L}(m)}} |M_{kn}| \cov_{\mu_\Lambda}(x_m, x_n)\\
	&= \sum_{\substack{k \notin B_{L}(m)\\ n\in B_{L}(m)\setminus B_{L/2}(m)}} |M_{kn}| \cov_{\mu_\Lambda}(x_m, x_n)\\
		&~~~~ + \sum_{\substack{k\notin B_{L}(m)\\ n\in B_{L/2}(m)}} |M_{kn}|\cov_{\mu_\Lambda}(x_m, x_n).\end{split}\end{equation}
The first sum in \eqref{JSecTerm} we bound exactly as we did in \eqref{JFirstTerm}.  Namely, we compute that
\begin{equation*}\begin{split}
\sum_{\substack{k \notin B_{L}(m)\\n\in B_{L}(m)\setminus B_{L/2}(m)}} |M_{kn}| \cov_{\mu_\Lambda}(x_m,x_n)
	&\lesssim \sum_{n\in B_{L}(m)\setminus B_{L/2}(m)}  \sum_{k\notin B_{L}(m)} \frac{|M_{kn}|}{L^{d+\alpha_{\ref{FWDC}}}}\\
	&\lesssim L^{d-\min\{1,\alpha_{\ref{e_alg_decay_f}}\}} L^{-d - \alpha_{\ref{FWDC}}} = L^{ - (\min\{1,\alpha_{\ref{e_alg_decay_f}}\} + \alpha_{\ref{FWDC}})}
.\end{split}\end{equation*}
The second sum we bound as follows
\[\begin{split}
\sum_{\substack{k \notin B_{L}(m)\\n\in B_{L/2}(m)}}& |M_{kn}|\cov_{\mu_\Lambda}(x_m, x_n)\\
	&\lesssim \sum_{n\in B_{L/2}(m)} \cov_{\mu_\Lambda}(x_m,x_n) \sum_{k \notin B_{L}(m)} |M_{kn}| \\
	&\lesssim \sum_{n\in B_{L/2}(m)} \frac{1}{(1+|n-m|)^{d+\alpha_{\ref{FWDC}}}}\sum_{k \notin B_{L}(m)} \frac{1}{(1 + |k-n|)^{d+\alpha_{\ref{e_alg_decay_f}}}}  \\
	&\lesssim \sum_{n\in B_{L/2}(m)} \frac{1}{(1+|n-m|)^{d+\alpha_{\ref{FWDC}}}} \frac{1}{L^{\alpha_{\ref{e_alg_decay_f}}}}\\
	&\lesssim L^{ - (\min\{1,\alpha_{\ref{FWDC}}\} + \alpha_{\ref{e_alg_decay_f}})}
.\end{split}\]
Hence we may choose $L$ large enough such that the sum in \eqref{JSecTerm} is smaller than $1/4$.  This finishes the proof of~\eqref{eq:J_small}, which in turn concludes the proof of Theorem~\ref{p_bootstrap}.



\section*{Acknowledgements}
The authors would like to thank the anonymous referee of~\cite{bblMenzNittka} for bringing the problem, Theorem~\ref{p_main_result}, to our attention.  The authors would like to thank Pietro Caputo for enlightening discussions on the role of averaging in controlling short range interactions.  The authors also thank Nobuo Yoshida for his email correspondence. Finally, the authors thank the referees of this article. Due to their detailed reports and remarks on the first manuscript, the authors were able to make substantial improvements to this article.  Part of this work was performed within the framework of the LABEX MILYON (ANR-10-LABX-0070) of Universit\'e de Lyon, within the program ``Investissements d'Avenir" (ANR-11-IDEX-0007) operated by the French National Research Agency (ANR).



\bibliographystyle{amsplain}
\bibliography{generalization_yoshida}

\end{document}